\documentclass[12pt]{amsart}
\usepackage{mlmodern}

\usepackage[utf8]{inputenc}
\usepackage[margin=1.2in]{geometry}
\usepackage{amsmath}
\usepackage{amsfonts}
\usepackage{mathtools}
\usepackage{float}
\usepackage{amsthm}
\usepackage{mathrsfs}
\usepackage{graphicx}
\usepackage[colorlinks=true, linkcolor=blue, citecolor=blue, urlcolor=blue, breaklinks=true]{hyperref}
\usepackage{thmtools}
\usepackage[capitalise]{cleveref}
\usepackage{tikz}
\usepackage{soul}
\usepackage{enumitem}

\newcommand\R{\mathbb{R}}

\newcommand{\s}{\mathfrak{S}}
\newcommand{\can}{\mathop{\mathsf{can}}}

\newcommand{\des}{\mathop{\mathsf{des}}}

\newcommand{\Des}{\mathop{\mathsf{Des}}}
\newcommand{\pk}{\mathop{\mathsf{pk}}}
\newcommand{\lpk}{\mathop{\mathsf{lpk}}}
\newcommand{\bpk}{\mathop{\mathsf{bpk}}}
\newcommand{\bounce}{\mathop{\mathsf{bounce}}}
\newcommand{\bperm}{\mathop{\mathsf{bperm}}}
\newcommand{\vperm}{\mathop{\mathsf{vperm}}}
\newcommand{\bcomp}{\mathop{\mathsf{bcomp}}}
\newcommand\cC{\mathcal{C}}

\newcommand{\CanDy}{\operatorname{CanDy}}
\newcommand{\Dyck}{\operatorname{Dyck}}

\newtheorem{theorem}{Theorem}[section]

\newtheorem{corollary}[theorem] {Corollary}
\newtheorem{definition}[theorem]{Definition}
\newtheorem{example}[theorem]{Example}
\newtheorem{lemma}[theorem]{Lemma}

\newtheorem{proposition}[theorem]{Proposition}
\newtheorem{remark}[theorem]{Remark}

\newtheorem*{utheorem}{Theorem}

\theoremstyle{definition}
\newtheorem{algorithm}[theorem]{Algorithm}

\title{How to bounce your canon permutation}

\author{Danai Deligeorgaki}
\address{\scriptsize Department of Mathematics, Universitat de Barcelona, Barcelona, Spain}
\email{deligeorgaki@ub.edu}

\author{Krishna Menon}
\address{\scriptsize Department of Mathematics, KTH Royal Institute of Technology, Stockholm, Sweden}
\email{puzhan@kth.se}

\date{\today}

\begin{document}

\begin{abstract}[old abstract]
    We study a new class of palindromic descent polynomials on multiset permutations. 
Given a Dyck path $d$ of semilength $n$ and a permutation $\sigma$ of size $n$, one can label the up-steps and down-steps of $d$ with the elements of $\sigma$. 
The labeled Dyck path gives rise to a multiset permutation called a {canon} (or nonnesting) permutation. 
Ranging over all Dyck paths for fixed $\sigma$  produces a collection of multiset permutations whose descent polynomial is a shift of the Narayana polynomial, while allowing both the Dyck path and the permutation to vary results in a descent polynomial that is a product of a Narayana and an Eulerian polynomial (Elizalde, 2024). 
In this article, we fix $d$ and range over all permutations of size $n$. 
We prove that the resulting descent polynomial is always palindromic, with its degree determined by the number of peaks in the bounce path of $d$. Our descent polynomials are not unimodal in general, but they are free from internal zeros.

In the process, we provide a method to traverse Dyck paths with the same bounce path by swapping valleys into peaks, while labeling their steps to maximize the number of descents. 

We further identify an unexpected connection between the canon permutations contributing to the leading coefficient of the descent polynomial and Dyck paths below 
$d$, yielding a sharp combinatorial lower bound via a natural compatibility condition on valleys.

\end{abstract}

\begin{abstract}
We study a new class of palindromic descent polynomials on multiset permutations. 
Given a Dyck path $d$ of semilength $n$ and a permutation $\sigma$ of size $n$, one can label the up-steps and down-steps of $d$ with the elements of $\sigma$. 
The labeled Dyck path gives rise to a multiset permutation called a {canon} (or nonnesting) permutation (Elizalde, 2024).

We consider the descent polynomial $C_d$ of canon permutations for fixed $d$, letting $\sigma$ vary over the permutations of size $n$. 
We show that $C_d$ is always palindromic, with degree determined by the number of peaks in the bounce path of 
$d$. Although $C_d$ is not unimodal in general, it is free from internal zeros.
We further establish a connection between the canon permutations contributing to the leading coefficient of $C_d$ and Dyck paths below $d$ in the Dyck lattice, yielding a sharp combinatorial lower bound for it.  Finally, we propose connections to known combinatorial sequences, including the Catalan numbers.

\textcolor{blue}{thinking about it again, a short sentence about this palindromicity phenomenon for canons etc could be nice}

\end{abstract}

\begin{abstract}

We study a new class of palindromic descent polynomials.
Given a Dyck path $d$ of semilength $n$ and a permutation $\sigma$ of $[n]:=\{1,\ldots,n\}$, one can label the up-steps and down-steps of $d$ with the elements of $\sigma$.
The labeled Dyck path gives rise to a multiset permutation called a {canon} (or nonnesting) permutation (Elizalde, 2024). These permutations come up as linear extensions as well as regions of certain hyperplane arrangements.
We consider the descent polynomial $C_d$ of canon permutations for fixed $d$, letting $\sigma$ vary over the permutations of size $n$.
Equivalently, $C_d$ is the descent polynomial of permutations of $[2n]$ whose odd and even entries correspond to up- and down-steps, respectively, and have the same relative order.
We show that $C_d$ is always palindromic, with degree determined by the number of peaks in the bounce path of $d$. 
This refines the palindromicity of the descent polynomial for all canon permutations (where $d$ and $\sigma$ can both vary) which is a product of an Eulerian and a Narayana polynomial. While $C_d$ is not unimodal in general, it is free from internal zeros. 
We also connect the canon permutations contributing to the leading coefficient of $C_d$ to Dyck paths below $d$ in the Dyck lattice, yielding a sharp combinatorial lower bound for it.  Finally, we propose connections to known combinatorial sequences, including the Catalan numbers.

\textcolor{blue}{I think last sentence unecessary, but we could rephrase and attach it with an "and" to the second to last sentence too. Need to shorten a bit I think :/}
\textcolor{red}{I've added another version below. I removed the odd,even $2n$ thing and restructured a bit. Parts we could delete I've marked in cyan (deletion would require slight rewording of other sentences). We could also break into two paragraphs: background + what we do.}

\end{abstract}

\begin{abstract}

We study a new class of palindromic descent polynomials. 
Given a Dyck path $d$ of semilength $n$ and a permutation $\sigma$ of size $n$, one can label the up-steps and down-steps of $d$ with the elements of $\sigma$. 
The labeled Dyck path gives rise to a multiset permutation called a {canon} (or nonnesting) permutation. 
{These permutations arise as linear extensions of posets as well as regions of hyperplane arrangements. }
The study of descents in canon permutations was initiated by Elizalde, {who showed that the descent polynomial for all canon permutations of a fixed length is a product of an Eulerian and a Narayana polynomial. }
We study a refinement of these polynomials.

To any Dyck path $d$, we associate the descent polynomial $C_d$ of all canon permutations obtained by labeling $d$. This polynomial can also be seen as the descent polynomial for a certain class of standard permutations. 
While $C_d$ is not unimodal in general, it is palindromic and free from internal zeros. 
We show that the degree of $C_d$ is determined by the number of peaks in the bounce path of $d$. 
We describe methods to generate canon permutations with this maximum number of descents. 
In particular, we associate maximizers to certain Dyck paths below $d$ in the Dyck lattice{, yielding a sharp combinatorial lower bound for the leading term in $C_d$}. 
{We end with several computational observations and directions for future research.}

\end{abstract}

\begin{abstract}

We study a new class of palindromic descent polynomials. 
Given a Dyck path $d$ of semilength $n$ and a permutation $\sigma$ of size $n$, one can label the up-steps and down-steps of $d$ with the elements of $\sigma$. 
The labeled Dyck path determines a multiset permutation called a {canon} (or nonnesting) permutation. Such permutations arise as linear extensions of posets and as regions of hyperplane arrangements.
Elizalde showed that the descent polynomial for all canon permutations of fixed length factors as a product of an Eulerian and a Narayana polynomial.

We refine these polynomials by associating to $d$ a descent polynomial $C_d$ over the canon permutations obtained from $d$.
 We prove that $C_d$
 is palindromic and free of internal zeros, though not unimodal in general.
Its degree is determined by the number of peaks in the bounce path of $d$. 
We establish a correspondence between canon permutations attaining the maximum number of descents and Dyck paths below $d$ in the Dyck lattice satisfying a valley condition. 
Each such path contributes a number of maximizers equal to the number of linear extensions of an associated poset, yielding a combinatorial interpretation of the leading coefficient of $C_d$.




\end{abstract}

\maketitle

\section{Introduction}
Canon permutations are a class of multiset permutations that exhibit rich combinatorial structure. 
For the multiset $M=\{1^m,2^m,\ldots,n^m\}$ with $m=2$, this class consists of exactly the multiset permutations that avoid the patterns $1221$ and $2112$. 
For general $m$, a canon permutation is a shuffle of $m$ copies of a permutation of size $n$ satisfying the following: the subsequence formed by the $j$-th copy of each element of $[n]:=\{1,2,\ldots,n\}$ is identical for all $j\in[m]$. 
For instance, $13123424$ is a canon permutation of $\{1^2,2^2,3^2,4^2\}$, with the sequence formed by the $j$-th copy being $1324$, whereas  $13123442$ is not a canon permutation. 
We denote by $\cC_n^{m}$ the set of all canon permutations for given $m,n$.

Although canon permutations (also called \textit{nonnesting} permutations for $m=2$) have been used in other contexts (for example, to label regions of hyperplane arrangements \cite{ber}), Sergi Elizalde \cite{sergi1} observed that their descent polynomial is of particular interest. Not only is it palindromic, but it further factorizes as a product of an Eulerian polynomial and a Narayana polynomial. This exceptional structure was later explained using standard Young tableaux by Elizalde in \cite{sergi2} and through labeled posets by Matthias Beck and Danai Deligeorgaki in \cite{danai}, where further generalizations were developed. 
Specifically, the descent polynomial corresponding to $\cC_n^{m}$ is the numerator of the rational generating function of the order polynomial for the ordinal sum of a product of two chains ($[m]\times [n]$) with an $n$-element antichain \cite{danai}.

The \emph{descent polynomial} for the canon permutations of the multiset $\{1^m,2^m,\ldots, n^m\}$ is denoted by \[C^{m}_n(t):=\sum_{\pi \in \cC_n^{m}}t^{\des(\pi)}, \]
where $j$ is a \emph{descent} of $\pi$ if $\pi(j+1) < \pi(j)$ and
$\des(\pi)$ is the number of descents of $\pi$. 
For $\sigma \in \s_n$, Elizalde \cite{sergi1} also considered the polynomial 
\[C^{m,\sigma}_n(t):=\sum_{\pi \in \cC_n^{m,\sigma}}t^{\des(\pi)},\] 
where $\cC_n^{m,\sigma}$ is the set of canon permutations such that the sequence formed by the $j$-th copy of each element of $[n]$ is $\sigma$. 
Narayana polynomials play a crucial role in this specialization. 
Let us denote by $N_n(t)$ the $n$-th Narayana polynomial, enumerating peaks in
Dyck paths. 
If we fix the underlying subsequence $\sigma$ to be the identity permutation,
$C^{2,\text{id}}_n(t)=N_n(t)$. 
More generally, for any $\sigma\in \s_n$, \[C^{2,\sigma}_n(t)=t^{\des(\sigma)}N_n(t).\] This was proved bijectively in \cite{sergi1} (case $m=2$) and generalized to all $m$ in \cite{sergi2} using standard Young tableaux. 
A shorter (but non-bijective) proof is also presented in \cite{danai} using $(P,\omega)$-partitions. 
The expression of $C^{2,\sigma}_n(t)$ as a shifted Narayana polynomial is fundamental for establishing distributional properties for $C^{2,\sigma}_n(t)$ (and, further, for $C^{2}_n(t)$). 
For instance, it follows that  $C^{2,\sigma}_n(t)$ is palindromic and its degree depends on $\sigma$. 

 In this article, we look at a different, natural restriction on canon permutations that is less understood, while focusing on the case $m=2$. In particular, we study canon permutations where the ``shuffle order'' is fixed and the subsequence $\sigma\in \s_n$ can vary. For example, $121323$ and $212313$ have the same shuffle order, since the first occurrences of the entries $\{1,2,3\}$ are in positions $1$, $2$ and $4$ in both permutations. In the case $m=2$, a shuffle ordering is captured by a Dyck path (see \cite[Section 3]{sergi1}, or alternatively \cite[Section 2]{danai}). 
Specifically, provided with a Dyck path $d$ of semilength $n $ and a permutation $\sigma \in \s_n$, we consider the canon permutation, denoted by $\can(d,\sigma)$, formed by labeling the up-steps (and the down-steps) of $d$ with the entries of $\sigma$, and then reading the permutation formed (from left to right). 
See Figure \ref{fig:canon123} for an example, and Section~\ref{sec: prel} for background on Dyck paths.
\bigskip
\begin{figure}[H]
    \centering
    \begin{tikzpicture}[scale=0.95]
      \draw[dotted] (0, 0) grid (6, 2);
      \draw[thick, color=blue] (0, 0) -- (1, 1) -- (2, 2) -- (3, 1) -- (4, 2) -- (5, 1) -- (6, 0) ;
        \node at (0 + 0.4, 0 + 0.75) {\scriptsize 1};
        \node at (1 + 0.4, 1 + 0.75) {\scriptsize 2};
        \node at (2 + 0.6, 2 - 0.25) {\scriptsize 1};
        \node at (3 + 0.4, 1 + 0.75) {\scriptsize 3};
        \node at (4 + 0.6, 2 - 0.25) {\scriptsize 2};
        \node at (5 + 0.6, 1 - 0.25) {\scriptsize 3};
                \node at (3, -0.5) {\scriptsize $\can(d, 123)=121323$};
    \end{tikzpicture} \qquad
 \begin{tikzpicture}[scale=0.95]
      \draw[dotted] (0, 0) grid (6, 2);
      \draw[thick, color=blue] (0, 0) -- (1, 1) -- (2, 2) -- (3, 1) -- (4, 2) -- (5, 1) -- (6, 0) ;
        \node at (0 + 0.4, 0 + 0.75) {\scriptsize 2};
        \node at (1 + 0.4, 1 + 0.75) {\scriptsize 1};
        \node at (2 + 0.6, 2 - 0.25) {\scriptsize 2};
        \node at (3 + 0.4, 1 + 0.75) {\scriptsize 3};
        \node at (4 + 0.6, 2 - 0.25) {\scriptsize 1};
        \node at (5 + 0.6, 1 - 0.25) {\scriptsize 3};
        \node at (3, -0.5) {\scriptsize $\can(d, 213)=212313$};
    \end{tikzpicture}
    
    \caption{Two canon permutations on the same Dyck path.}
    \label{fig:canon123}
\end{figure}
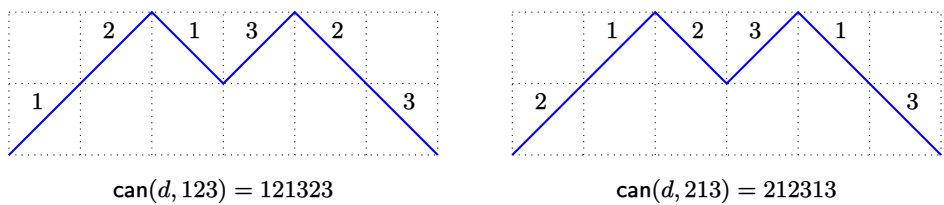

To a Dyck path $d$ of semilength $n$, we associate the polynomial
\begin{equation*}
    C_d(t) \coloneqq \sum_{\sigma \in \s_n} t^{\des (\can (d, \sigma))},
\end{equation*}
called the \emph{canon descent polynomial of $d$}. Summing over the set of Dyck paths of semilength $n$, we recover the set of canon permutations for $\{1^2, 2^2,\ldots, n^2 \}$, and 
\[  \sum_{d \in \Dyck_n} C_d(t) = C^{2}_n(t).\]

Our main contribution is the following theorem (summarizing Propositions \ref{sym} and \ref{niz}, Theorem \ref{md} and Corollary \ref{LEplusLB}).

\begin{utheorem}
  For any Dyck path $d$ of semilength $n$, the canon descent polynomial  of $d$ is
    \begin{enumerate}
      \item[a)] \textit{palindromic} with respect to degree $2n - 1 - a$, where $a$ is the number of \emph{low peaks of $d$},
      \item[b)] of degree $2n - 1 - b$, where $b$ is the number of peaks of the \emph{bounce path} of $d$, and
      \item[c)] free from internal zeros.
  \end{enumerate}
Moreover, its leading coefficient 
counts Dyck paths $b$ below $d$ satisfying a compatibility condition, 
each weighted by the number of linear extensions of a poset 
associated with  the pair 
$(b,d)$.
\end{utheorem}

One reason the distributional properties of the polynomials $C_d$ are interesting is because they refine the descent polynomial of canon permutations, which is the product of an Eulerian and a Narayana polynomial. As the theorem above states, palindromicity and the no-internal-zeros property are inherited, but unimodality and real-rootedness are not (\Cref{notuni}). 
The polynomials $C_d$ can also be thought of as descent polynomials of standard permutations. Given a Dyck path $d$ of semilength $n$, we label its up-steps (resp.\ down-steps) with the odd (resp.\ even) numbers in $[2n]$, with the same relative order in each parity. 
Now, $C_d$ is the resulting descent polynomial, ranging over all such permutations.

\subsection*{Summary}

We present our definitions along with some background on Dyck paths in Section \ref{sec: prel}. 
Then, in Section \ref{sec:symdeg} we prove the palindromicity of the polynomials $C_d$, and compute their degree. 
In particular, using the bounce path of $d$, we construct a labeling that maximizes descents over the canon permutations associated to $d$. 
A natural candidate for a descent-maximizer is the decreasing permutation $\delta_n:= n \cdots 21\in \s_n$, which turns out to be the unique maximizer for bounce paths. 
However, in Section \ref{sec: more distributional properties}, we show that $\delta_n$ is a descent maximizer if and only if the number of peaks of $d$ equals the number of peaks of its corresponding bounce path -- these paths are counted by \cite[\href{https://oeis.org/A1519}{A1519}]{oeis}. 
Section \ref{sec: more distributional properties} also contains the proof that $C_d$ is free from internal zeros. 

Sections \ref{sec:valleys} and \ref{sec: interpret leader} are concerned with identifying all descent-maximizers for $d$, i.e., interpreting the leading coefficient of $C_d$. 
In Section \ref{sec:valleys}, we provide a method to traverse Dyck paths with the same underlying bounce path, while labeling their steps to maximize descents. 
In Section \ref{sec: interpret leader}, we extend this method and associate a Dyck path under $d$ in the Dyck lattice to each descent-maximizer of $d$. 
This provides a lower bound for the leading coefficient of $C_d$. 
To compute the exact leading coefficient, we sum over all suitable Dyck paths, weighting each path by the number of linear extensions of a poset determined by its interaction with $d$. 
We end with \Cref{sec:future} where we propose future directions, including connections to known OEIS sequences.

\section{Preliminaries}\label{sec: prel}

A \emph{Dyck path} of semilength $n$ is a lattice path in $\R^2$ that starts at the origin $(0, 0)$, ends at $(2n, 0)$, has steps $U = (1, 1)$ (up-steps) or $D = (1, -1)$ (down-steps), and never falls below the $x$-axis. 
We use $\Dyck_n$ to denote the set of Dyck paths of semilength $n$. 
A Dyck path of semilength $6$ is shown in \Cref{fig:dyck}.

A \emph{peak} of a Dyck path is an up-step that is immediately followed by a down-step, i.e., an occurrence of $UD$. 
We denote the number of peaks in a Dyck path $d$ by  $\pk d$. 
We also use $\lpk d$ to denote the number of \emph{low peaks} of $d$, that is, the peaks whose steps touch the $x$-axis.

\begin{figure}[H]
    \centering
    \begin{tikzpicture}[scale=0.95]
      \draw[dotted] (0, 0) grid (12, 2);
      \draw[thick, color=blue] (0, 0) -- (1, 1) -- (2, 2) -- (3, 1) -- (4, 2) -- (5, 1) -- (6, 0) -- (7, 1) -- (8, 0) -- (9, 1) -- (10, 2) -- (11, 1) -- (12, 0);
    \end{tikzpicture}
    \caption{An element $d \in \Dyck_6$ with $\pk d = 4$ and $\lpk d = 1$.}
    \label{fig:dyck}
\end{figure}
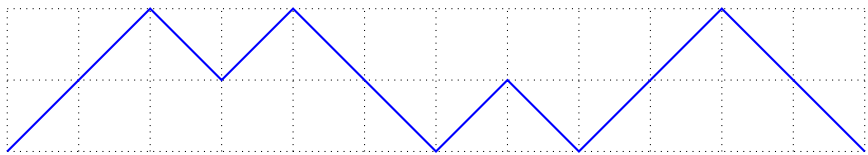

A crucial role among Dyck paths is played by the primitive paths, often used in path decompositions and recursive constructions.
A Dyck path of semilength $n$ is said to be \emph{primitive} if it only touches the $x$-axis at $(0, 0)$ and $(2n, 0)$. 
Any Dyck path can be broken into \emph{primitive factors} using the points where it touches the $x$-axis. 
For example, the Dyck path in \Cref{fig:dyck} has $3$ primitive factors: $UUDUDD$, $UD$, and $UUDD$.

This work focuses on a polynomial obtained by fixing a Dyck path and varying over a set of permutations. 
In particular, provided with $d \in \Dyck_n$ and a permutation $\sigma \in \s_n$, we will denote by $\can(d, \sigma)$ the canon permutation obtained by replacing the up-steps (similarly, down-steps) in $d$ with the terms of $\sigma$ in order. 
For instance, if $d$ is the Dyck path in \Cref{fig:dyck}, then $\can(d, 541623) = 545141662323$ as shown in \Cref{fig:canon}.

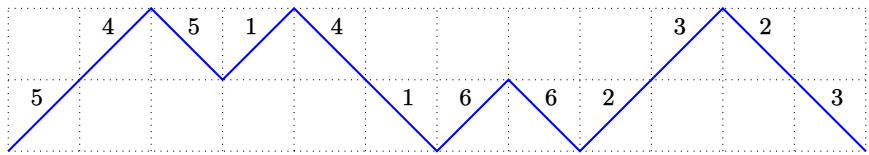
\begin{figure}[H]
    \centering
    \begin{tikzpicture}[scale=0.95]
      \draw[dotted] (0, 0) grid (12, 2);
      \draw[thick, color=blue] (0, 0) -- (1, 1) -- (2, 2) -- (3, 1) -- (4, 2) -- (5, 1) -- (6, 0) -- (7, 1) -- (8, 0) -- (9, 1) -- (10, 2) -- (11, 1) -- (12, 0);
        \node at (0 + 0.4, 0 + 0.75) {\scriptsize 5};
        \node at (1 + 0.4, 1 + 0.75) {\scriptsize 4};
        \node at (2 + 0.6, 2 - 0.25) {\scriptsize 5};
        \node at (3 + 0.4, 1 + 0.75) {\scriptsize 1};
        \node at (4 + 0.6, 2 - 0.25) {\scriptsize 4};
        \node at (5 + 0.6, 1 - 0.25) {\scriptsize 1};
        \node at (6 + 0.4, 0 + 0.75) {\scriptsize 6};
        \node at (7 + 0.6, 1 - 0.25) {\scriptsize 6};
        \node at (8 + 0.4, 0 + 0.75) {\scriptsize 2};
        \node at (9 + 0.4, 1 + 0.75) {\scriptsize 3};
        \node at (10 + 0.6, 2 - 0.25) {\scriptsize 2};
        \node at (11 + 0.6, 1 - 0.25) {\scriptsize 3};
    \end{tikzpicture}
    \caption{Labeling the steps of a Dyck path using the permutation $541623$.}
    \label{fig:canon}
\end{figure}

Notice that, for a fixed Dyck path $d$, there is a down-step {corresponding} to an up-step (or vice-versa) that shares the same label in $\can(d, \sigma)$, regardless of the choice of the permutation $\sigma$. We will take advantage of this property, referring to them as the \emph{corresponding} up or down-step.
In terms of the Dyck path, the down-step corresponding to the $i$-th up-step is simply the $i$-th down-step.

Our goal is to study statistics on canon permutations obtained for a fixed Dyck path, in particular their descent polynomials, following the line of research in \cite{danai,sergi1,sergi2}. 
For a (multiset) permutation $\pi = \pi_1\pi_2 \cdots \pi_k$, we define $\Des(\pi) = \{i \in [k - 1] \mid \pi_i > \pi_{i + 1}\}$, the descent set of $\pi$. 
We denote the number of descents in $\pi$ by $\des(\pi)$, i.e., $\des(\pi) = |\Des(\pi)|$. 
To simplify notation, to denote the number of descents of a canon permutation $\can(d, \sigma)$, we simply write $\des(d, \sigma)$ instead of $\des(\can(d, \sigma))$. 
Similarly, we use $\Des(d, \sigma)$ to denote $\Des(\can(d, \sigma))$.
We can now define our main objects of interest. 
To a Dyck path $d$ of semilength $n$, we associate the polynomial
\begin{equation*}
    C_d(t) \coloneqq \sum_{\sigma \in \s_n} t^{\des (d, \sigma)},
\end{equation*}
called the \emph{canon descent polynomial of $d$}.
For example, if $d = UUDUDD$, then varying $\sigma$ we obtain the canon permutations $121323,\ 131232,\ 212313,\ 232131,\ 313212,\ 323121$ and hence $C_d(t) = 3t^2 + 3t^3$. 
In what follows, our objective will be to obtain results about these polynomials.

\begin{remark}
    We do not keep track of plateaus in this polynomial (like in \cite{danai,sergi1,sergi2}) since plateaus of canon permutations only depend on the choice of Dyck path and hence factor out. In other words, the corresponding {weak descent (non-ascent) polynomial} for a fixed Dyck path is just a shift of $C_d(t)$.
\end{remark}

Some characteristics, such as the degree of $C_d$, are captured by a significant class of Dyck paths called bounce paths. 
For $d \in \Dyck_n$, its \emph{bounce path} is constructed as follows: 
Start at the point $(0, 0)$ and keep moving using up-steps until encountering a down-step of $d$. 
Then keep moving using down-steps until encountering the $x$-axis. 
Then, again, keep moving using up-steps until encountering a down-step of $d$, and repeat the process until $(2n, 0)$ is reached. 
We denote the bounce path of $d$ by $\bounce d$. 
An example is shown in \Cref{fig:bouncepath}.

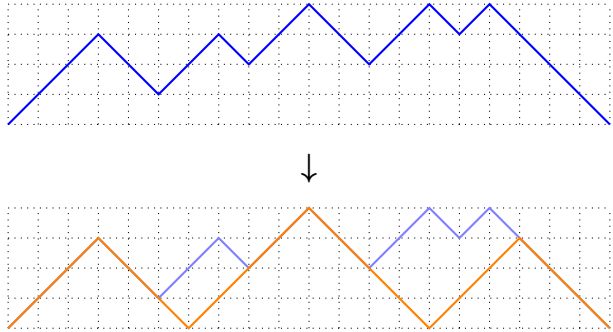
\begin{figure}[H]
    \centering
    \begin{tikzpicture}[scale = 0.4]
        \draw[dotted] (0, 0) grid (20, 4);
        \draw[thick, blue] (0, 0) -- (1, 1) -- (2, 2) -- (3, 3) -- (4, 2) -- (5, 1) -- (6, 2) -- (7, 3) -- (8, 2) -- (9, 3) -- (10, 4) -- (11, 3) -- (12, 2) -- (13, 3) -- (14, 4) -- (15, 3) -- (16, 4) -- (17, 3) -- (18, 2) -- (19, 1) -- (20, 0);
    \end{tikzpicture}\\[0.1cm]
    $\downarrow$\\[0.3cm]
    \begin{tikzpicture}[scale = 0.4]
        \draw[dotted] (0, 0) grid (20, 4);
        \draw[thick, blue!50] (0, 0) -- (1, 1) -- (2, 2) -- (3, 3) -- (4, 2) -- (5, 1) -- (6, 2) -- (7, 3) -- (8, 2) -- (9, 3) -- (10, 4) -- (11, 3) -- (12, 2) -- (13, 3) -- (14, 4) -- (15, 3) -- (16, 4) -- (17, 3) -- (18, 2) -- (19, 1) -- (20, 0);
        \draw[thick, orange] (0, 0) -- (3, 3) -- (6, 0) -- (10, 4) -- (14, 0) -- (17, 3) -- (20, 0);
    \end{tikzpicture}
    \caption{A Dyck path (\textcolor{blue}{top}) and its bounce path (\textcolor{orange}{bottom}).}
    \label{fig:bouncepath}
\end{figure}

Bounce paths were introduced by James Haglund in the study of $q,t$-Catalan numbers (see \cite[Chapter 3]{hanglundbounce}), and have since played a key role in understanding the structure and statistics of Dyck paths and other Catalan families.
They have the special property that they are
determined by a composition of $n$. This is just the sequence of heights of the peaks of the bounce path. In fact, we can attach such a composition to any Dyck path, by computing it for the corresponding bounce path. Given a Dyck path $d$, we denote this composition by $\bcomp d$. 
For example, if $d$ is the Dyck path in \Cref{fig:bouncepath}, then $\bcomp d = (3, 4, 3)$.


\section{Symmetry and degree}\label{sec:symdeg}

In this section, we show bijectively that the polynomial $C_d(t)$ is symmetric for any Dyck path $d$. We then compute the degree of  $C_d(t)$ using the number of peaks in the bounce path of $d$.

For a polynomial $f(t)$, we use $[t^i]f(t)$ to denote the coefficient of $t^i$ in $f(t)$.

\begin{proposition}\label{sym}
    Let $d \in \Dyck_n$ and $k = 2n - 1 - \lpk d$. 
    For any $i \in [0, k]$, we have $[t^i]C_d(t) = [t^{k- i}] C_d(t)$.
\end{proposition}

\begin{proof}
    For a permutation $\sigma = \sigma_1 \cdots \sigma _n \in \s_n$, its \emph{complement} is the permutation $\sigma^c = (n + 1 - \sigma_1) \cdots (n + 1 - \sigma_n)$. 
    For any $\sigma \in \s_n$, 
    \[\des(d,\sigma) + \des(d,\sigma^c) = (2n - 1) - \lpk d.\]
    This is because each of the $(2n - 1)$ possible descent positions in the canon permutation $\can(d,\sigma)$ is exactly one of the following
    \begin{itemize}
        \item a descent in $\can(d, \sigma)$,

        \item a descent in $\can(d, \sigma^c)$, or

        \item a plateau: two consecutive terms being equal (low peaks of $d$).
    \end{itemize}
    The result now follows since the complement map is an involution on permutations.
\end{proof}

\begin{example}
    The two Dyck paths in \Cref{fig:symdycks} both have semilength $4$ and one low peak each.
    \begin{enumerate}
        \item For $d = UDUUUDDD$, we have 
        \begin{equation*}
            C_d(t) = t + 8t^2 + 6t^3 + 8t^4 + t^5.
        \end{equation*}
        Here, $[t^i]C_d(t) = [t^{6 - i}]C_d(t)$ for all $i \in [0, 6]$.
        
        \item For $d = UUDUDDUD$, we have 
        \begin{equation*}
            C_d(t) = 4t^2 + 16t^3 + 4t^4.
        \end{equation*}
        This polynomial also satisfies the symmetry condition. 
        Note, however, that it has a different degree than the previous one despite both Dyck paths having the same semilength and number of low peaks.
    \end{enumerate}
\end{example}

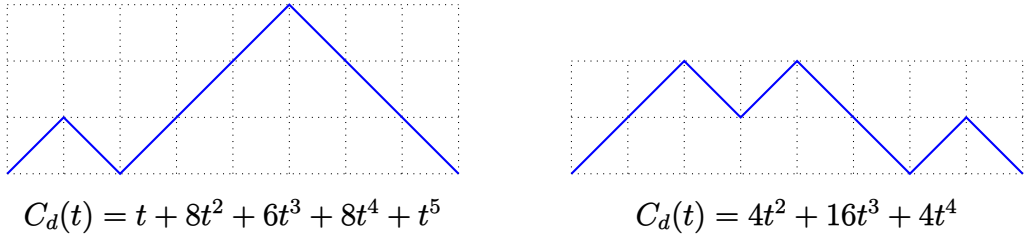
\begin{figure}[H]
    \centering
    \begin{tikzpicture}[scale =0.75]
        \draw[dotted] (0, 0) grid (8, 3);
        \draw[thick, blue] (0, 0) -- (1, 1) -- (2, 0) -- (3, 1) -- (4, 2) -- (5, 3) -- (6, 2) -- (7, 1) -- (8, 0);
        \node at (4, -0.75) {$C_d(t) = t + 8t^2 + 6t^3 + 8t^4 + t^5$};
        \begin{scope}[xshift = 10cm]
            \draw[dotted] (0, 0) grid (8, 2);
            \draw[thick, blue] (0, 0) -- (1, 1) -- (2, 2) -- (3, 1) -- (4, 2) -- (5, 1) -- (6, 0) -- (7, 1) -- (8, 0);
            \node at (4, -0.75) {$C_d(t) = 4t^2 + 16t^3 + 4t^4$};
        \end{scope}
    \end{tikzpicture}
    \caption{Two Dyck paths of semilength $4$ with one low peak each.}
    \label{fig:symdycks}
\end{figure}

\begin{remark}\label{notuni}
    Although we will see that the polynomials $C_d(t)$ are free from internal zeroes in \Cref{sec: more distributional properties}, they need not be real-rooted nor unimodal. 
    For $d = U^3D^3$,we have $C_d(t) = t + 2t^2 + 2t^3 + t^4$, which is not real rooted. 
    For $d = UDU^3D^3$, we have $C_d(t) = t + 8t^2 + 6t^3 + 8t^4 + t^5$, which is not unimodal.
\end{remark}

We now turn to computing the degree of $C_d(t)$.

\begin{definition}
    For any Dyck path $d$, we set $m_d \coloneqq \deg C_d(t)$.
\end{definition}
Hence, $m_d$ is the maximum possible number of descents that a canon permutation associated to $d$ can have. 
Since we will be using the peaks of $\bounce d$ to compute $m_d$, we use the following notation for convenience: 
For any Dyck path $d$, we use $\bpk d$ to denote the number of peaks in $\bounce d$. 
That is, $\bpk d = \pk(\bounce d)$.

\begin{lemma}\label{upperbound}
    For any $d \in \Dyck_n$, we have
    \begin{equation*}
        m_d \leq 2n - 1 - \bpk d.
    \end{equation*}
\end{lemma}

\begin{proof}
    Let $\sigma \in \s_n$ and $k = \bpk d$. 
    We will show that $\can(d, \sigma)$ has a subsequence of the form $a_1a_1 a_2a_2 \cdots a_ka_k$. 
    This will prove the result since between any two equal terms, there must be a non-descent. 
    In fact, the peaks of the bounce path pick out the left-most such subsequence of $\can(d, \sigma)$.

    Let $a_i$ be the label on the down-step of $d$ that coincides with the down-step of the $i$-th peak of $\bounce d$. 
    For example, if $\sigma$ is the identity permutation and $d$ is as in \Cref{fig:upperbound}, then $a_1 = 1$, $a_2 = 5$, and $a_3 = 8$.

    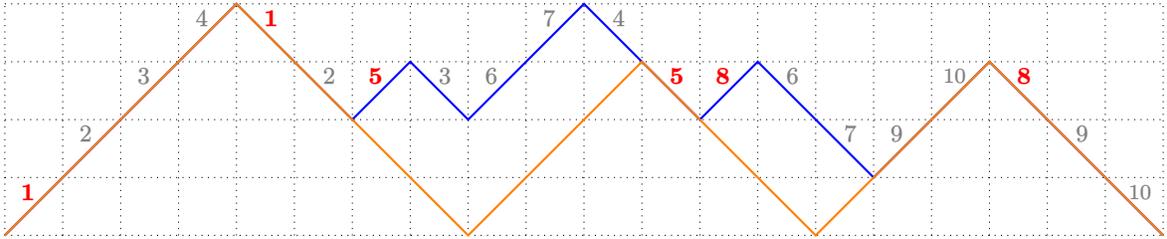
\begin{figure}[H]
        \centering
        \begin{tikzpicture}[scale = 0.77]
            \draw[dotted] (0, 0) grid (20, 4);
            \draw[thick, blue] (0, 0) -- (1, 1) -- (2, 2) -- (3, 3) -- (4, 4) -- (5, 3) -- (6, 2) -- (7, 3) -- (8, 2) -- (9, 3) -- (10, 4) -- (11, 3) -- (12, 2) -- (13, 3) -- (14, 2) -- (15, 1) -- (16, 2) -- (17, 3) -- (18, 2) -- (19, 1) -- (20, 0);
            \node at (0 + 0.4, 0 + 0.75) {\scriptsize \color{red}\textbf{1}};
            \node[text opacity = 0.5] at (1 + 0.4, 1 + 0.75) {\scriptsize 2};
            \node[text opacity = 0.5] at (2 + 0.4, 2 + 0.75) {\scriptsize 3};
            \node[text opacity = 0.5] at (3 + 0.4, 3 + 0.75) {\scriptsize 4};
            \node at (4 + 0.6, 4 - 0.25) {\scriptsize \color{red}\textbf{1}};
            \node[text opacity = 0.5] at (5 + 0.6, 3 - 0.25) {\scriptsize 2};
            \node at (6 + 0.4, 2 + 0.75) {\scriptsize \color{red}\textbf{5}};
            \node[text opacity = 0.5] at (7 + 0.6, 3 - 0.25) {\scriptsize 3};
            \node[text opacity = 0.5] at (8 + 0.4, 2 + 0.75) {\scriptsize 6};
            \node[text opacity = 0.5] at (9 + 0.4, 3 + 0.75) {\scriptsize 7};
            \node[text opacity = 0.5] at (10 + 0.6, 4 - 0.25) {\scriptsize 4};
            \node at (11 + 0.6, 3 - 0.25) {\scriptsize \color{red}\textbf{5}};
            \node at (12 + 0.4, 2 + 0.75) {\scriptsize \color{red}\textbf{8}};
            \node[text opacity = 0.5] at (13 + 0.6, 3 - 0.25) {\scriptsize 6};
            \node[text opacity = 0.5] at (14 + 0.6, 2 - 0.25) {\scriptsize 7};
            \node[text opacity = 0.5] at (15 + 0.4, 1 + 0.75) {\scriptsize 9};
            \node[text opacity = 0.5] at (16 + 0.4, 2 + 0.75) {\tiny 10};
            \node at (17 + 0.6, 3 - 0.25) {\scriptsize \color{red}\textbf{8}};
            \node[text opacity = 0.5] at (18 + 0.6, 2 - 0.25) {\scriptsize 9};
            \node[text opacity = 0.5] at (19 + 0.6, 1 - 0.25) {\tiny 10};

            \draw [thick, orange] (0, 0) -- (1, 1) -- (2, 2) -- (3, 3) -- (4, 4) -- (5, 3) -- (6, 2) -- (7, 1) -- (8, 0) -- (9, 1) -- (10, 2) -- (11, 3) -- (12, 2) -- (13, 1) -- (14, 0) -- (15, 1) -- (16, 2) -- (17, 3) -- (18, 2) -- (19, 1) -- (20, 0);
        \end{tikzpicture}
        \caption{Subsequence corresponding to peaks of the bounce path. \textcolor{blue}{On top}, the Dyck path $d$, with the identity permutation canon labeling. \textcolor{orange}{Under it}, the bounce path $\bounce d$. The \textcolor{red}{selected} numbers on $d$ come from the first up-step and down-step of each component in $\bounce d$. The subsequence of $\can(d, \sigma)$ with $k=3$ plateaus formed is $115588$.} 
        \label{fig:upperbound}
    \end{figure}

    We claim that $a_1a_1 a_2a_2 \cdots a_ka_k$ is a subsequence of $\can(d, \sigma)$. 
    We just have to show that for any $i \in [k - 1]$, the up-step labeled $a_{i + 1}$ appears after the down-step labeled $a_i$. 
    This follows from the definition of the bounce path. 
    In fact, for any $i \in [k - 1]$, the up-step labeled $a_{i + 1}$ is the \emph{first} up-step after the down-step labeled $a_i$.
\end{proof}

We will prove that the upper bound in Lemma \ref{upperbound} is in fact the value of $m_d$. 
To do this, we provide an algorithm that constructs a canon permutation with the required number of descents. 
That is, we start with a Dyck path $d$ and construct a canon permutation $\can(d, \sigma)$ such that $\des(d, \sigma) = 2n - 1 - \bpk d$.

The idea of this algorithm is as follows. 
For ordinary permutations, the decreasing permutation $\sigma = n(n-1)\cdots 1$ maximizes the number of descents. 
It is easy to verify that the decreasing permutation does not always maximize descents in the setting of canon permutations. 
Nevertheless, our approach to constructing a canon permutation with the maximum number of descents aims to mimic the behavior of a decreasing permutation. 
The guiding principle is to start labeling the steps of the Dyck path and attempt to decrease the label by $1$ as we move from one step to the next. 
The algorithm is described below.

\begin{algorithm}\label{algo}
    Let $d$ be a Dyck path of semilength $n$. 
    Label the first up-step and its corresponding (i.e., the first) down-step with $n$. 
    Repeat the following until all steps of $d$ have been labeled:
    \begin{itemize}
        \item Find the first unlabeled step in $d$ (this will be an up-step). 
        Suppose that $i$ is the label of the step immediately before it.
        \item If there is no step that already has the label $i - 1$, then label the first unlabeled step and its corresponding down-step with $i - 1$.
        \item If there are already steps labeled $i - 1$, then reduce by $1$ all existing labels whose value is at most $i - 1$. 
        Now label the first unlabeled step and its corresponding down-step with $i - 1$.
    \end{itemize}
\end{algorithm}

Essentially, we give precedence to giving the first unlabeled step a label exactly $1$ less than the step before it and we change the existing labels that would make this an invalid canon labeling. 
It is not difficult to see that the algorithm does indeed give a canon labeling of the Dyck path. 
An example of a step in this algorithm is shown in \Cref{fig:algorithmstep}.

\begin{figure}[H]
    \centering
    \begin{tikzpicture}[scale = 0.8]
        \node[opacity = 0] at (0, 3) {8};
        \node[opacity = 0] at (14, 0) {8};
        \draw[dotted] (0, 0) grid (14, 3);
        \draw[thick, blue] (0, 0) -- (1, 1) -- (2, 2) -- (3, 3) -- (4, 2) -- (5, 1) -- (6, 2) -- (7, 1) -- (8, 2) -- (9, 1) -- (10, 2) -- (11, 3) -- (12, 2) -- (13, 1) -- (14, 0);
        \node at (0 + 0.4, 0 + 0.75) {\scriptsize 7};
        \node at (1 + 0.4, 1 + 0.75) {\scriptsize 6};
        \node at (2 + 0.4, 2 + 0.75) {\scriptsize 4};
        \node at (3 + 0.6, 3 - 0.25) {\scriptsize 7};
        \node at (4 + 0.6, 2 - 0.25) {\scriptsize 6};
        \node at (5 + 0.4, 1 + 0.75) {\scriptsize 5};
        \node at (6 + 0.6, 2 - 0.25) {\scriptsize 4};
        \node at (7 + 0.4, 1 + 0.75) {\scriptsize 3};
        \node at (8 + 0.6, 2 - 0.25) {\scriptsize 5};
        \node at (11 + 0.6, 3 - 0.25) {\scriptsize 3};
    \end{tikzpicture}\\[-0.2cm]
    $\downarrow$\\
    \begin{tikzpicture}[scale = 0.8]
        \node[opacity = 0] at (0, 3) {8};
        \node[opacity = 0] at (14, 0) {8};
        \draw[dotted] (0, 0) grid (14, 3);
        \draw[thick, blue] (0, 0) -- (1, 1) -- (2, 2) -- (3, 3) -- (4, 2) -- (5, 1) -- (6, 2) -- (7, 1) -- (8, 2) -- (9, 1) -- (10, 2) -- (11, 3) -- (12, 2) -- (13, 1) -- (14, 0);
        \node at (0 + 0.4, 0 + 0.75) {\scriptsize 7};
        \node at (1 + 0.4, 1 + 0.75) {\scriptsize 6};
        \node[circle, draw = red, inner sep = 1pt] at (2 + 0.4, 2 + 0.75) {\scriptsize 4};
        \node at (3 + 0.6, 3 - 0.25) {\scriptsize 7};
        \node at (4 + 0.6, 2 - 0.25) {\scriptsize 6};
        \node at (5 + 0.4, 1 + 0.75) {\scriptsize 5};
        \node[circle, draw = red, inner sep = 1pt] at (6 + 0.6, 2 - 0.25) {\scriptsize 4};
        \node[circle, draw = red, inner sep = 1pt] at (7 + 0.4, 1 + 0.75) {\scriptsize 3};
        \node at (8 + 0.6, 2 - 0.25) {\scriptsize 5};
        \node at (9 + 0.4, 1 + 0.75) {\scriptsize \color{red}\textbf{4}};
        \node[circle, draw = red, inner sep = 1pt] at (11 + 0.6, 3 - 0.25) {\scriptsize 3};
        \node at (12 + 0.6, 2 - 0.25) {\scriptsize \color{red}\textbf{4}};
    \end{tikzpicture}\\[-0.2cm]
    $\downarrow$\\
    \begin{tikzpicture}[scale = 0.8]
        \node[opacity = 0] at (0, 3) {8};
        \node[opacity = 0] at (14, 0) {8};
        \draw[dotted] (0, 0) grid (14, 3);
        \draw[thick, blue] (0, 0) -- (1, 1) -- (2, 2) -- (3, 3) -- (4, 2) -- (5, 1) -- (6, 2) -- (7, 1) -- (8, 2) -- (9, 1) -- (10, 2) -- (11, 3) -- (12, 2) -- (13, 1) -- (14, 0);
        \node at (0 + 0.4, 0 + 0.75) {\scriptsize 7};
        \node at (1 + 0.4, 1 + 0.75) {\scriptsize 6};
        \node at (2 + 0.4, 2 + 0.75) {\scriptsize 3};
        \node at (3 + 0.6, 3 - 0.25) {\scriptsize 7};
        \node at (4 + 0.6, 2 - 0.25) {\scriptsize 6};
        \node at (5 + 0.4, 1 + 0.75) {\scriptsize 5};
        \node at (6 + 0.6, 2 - 0.25) {\scriptsize 3};
        \node at (7 + 0.4, 1 + 0.75) {\scriptsize 2};
        \node at (8 + 0.6, 2 - 0.25) {\scriptsize 5};
        \node at (9 + 0.4, 1 + 0.75) {\scriptsize 4};
        \node at (11 + 0.6, 3 - 0.25) {\scriptsize 2};
        \node at (12 + 0.6, 2 - 0.25) {\scriptsize 4};
    \end{tikzpicture}
    \caption{A step of \Cref{algo}.}
    \label{fig:algorithmstep}
\end{figure}

\begin{definition}
    For $d \in \Dyck_n$, we set $\bperm d$ to be the permutation in $\s_n$ such that $\can(d, \bperm d)$ is the canon permutation obtained when \Cref{algo} is applied to $d$.
\end{definition}
SageMath \cite{sagemath} code implementing \Cref{algo} to return $\bperm d$ can be found at
\begin{center}
    \href{https://github.com/KrishnaMenon1998/Bouncing-canons}{\texttt{https://github.com/KrishnaMenon1998/Bouncing-canons}}.
\end{center}

\begin{example}
    Completing \Cref{algo} on the Dyck path $d$ shown in \Cref{fig:algorithmstep} will result in the labeling shown in \Cref{fig:completedalgo}. 
    Hence $\bperm d = 7625143$.
\end{example}

\begin{figure}[H]
    \centering
    \begin{tikzpicture}[scale = 0.8]
        \node[opacity = 0] at (0, 3) {8};
        \node[opacity = 0] at (14, 0) {8};
        \draw[dotted] (0, 0) grid (14, 3);
        \draw[thick, blue] (0, 0) -- (1, 1) -- (2, 2) -- (3, 3) -- (4, 2) -- (5, 1) -- (6, 2) -- (7, 1) -- (8, 2) -- (9, 1) -- (10, 2) -- (11, 3) -- (12, 2) -- (13, 1) -- (14, 0);
        \node at (0 + 0.4, 0 + 0.75) {\scriptsize 7};
        \node at (1 + 0.4, 1 + 0.75) {\scriptsize 6};
        \node at (2 + 0.4, 2 + 0.75) {\scriptsize 2};
        \node at (3 + 0.6, 3 - 0.25) {\scriptsize 7};
        \node at (4 + 0.6, 2 - 0.25) {\scriptsize 6};
        \node at (5 + 0.4, 1 + 0.75) {\scriptsize 5};
        \node at (6 + 0.6, 2 - 0.25) {\scriptsize 2};
        \node at (7 + 0.4, 1 + 0.75) {\scriptsize 1};
        \node at (8 + 0.6, 2 - 0.25) {\scriptsize 5};
        \node at (9 + 0.4, 1 + 0.75) {\scriptsize 4};
        \node at (10 + 0.4, 2 + 0.75) {\scriptsize 3};
        \node at (11 + 0.6, 3 - 0.25) {\scriptsize 1};
        \node at (12 + 0.6, 2 - 0.25) {\scriptsize 4};
        \node at (13 + 0.6, 1 - 0.25) {\scriptsize 3};
    \end{tikzpicture}
    \caption{A Dyck path $d$ labeled with $\bperm d$.}
    \label{fig:completedalgo}
\end{figure}
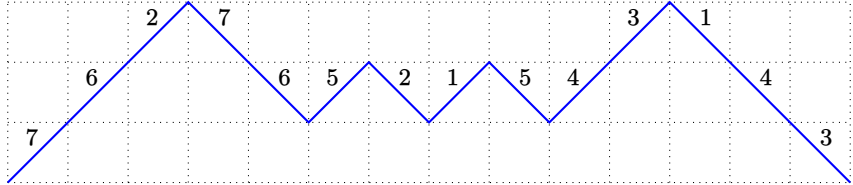

The notation `$\bperm$' is due to the fact that this labeling of the Dyck path will allow us to prove the following `bounce' expression for $m_d$. 
We also mention that this expression for $m_d$ was conjectured by FindStat \cite{FindStat}.

\begin{theorem}\label{md}
    For $d \in \Dyck_n$, we have
    \begin{equation*}
        m_d = 2n - 1 - \bpk d.
    \end{equation*}
\end{theorem}

Before presenting the proof, we make the following definition which will also be used in the sequel.

\begin{definition}
    The peaks of $\bounce d$ break $d$ into \emph{bounce factors}. 
    The first bounce factor consists of all steps before the down-step of the first peak of $\bounce d$ (which is also a down-step of $d$). 
    The second bounce factor consists of all steps before the down-step of the second peak of $\bounce d$, excluding those in the first bounce factor. 
    The other factors are defined similarly (see \Cref{fig:bouncefactors}).
\end{definition}

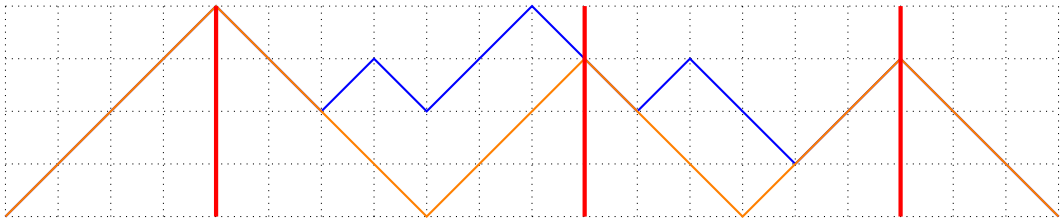
\begin{figure}[H]
    \centering
    \begin{tikzpicture}[scale = 0.7]
        \draw[dotted] (0, 0) grid (20, 4);
        \draw[thick, blue] (0, 0) -- (1, 1) -- (2, 2) -- (3, 3) -- (4, 4) -- (5, 3) -- (6, 2) -- (7, 3) -- (8, 2) -- (9, 3) -- (10, 4) -- (11, 3) -- (12, 2) -- (13, 3) -- (14, 2) -- (15, 1) -- (16, 2) -- (17, 3) -- (18, 2) -- (19, 1) -- (20, 0);

        \draw[thick, orange] (0, 0) -- (1, 1) -- (2, 2) -- (3, 3) -- (4, 4) -- (5, 3) -- (6, 2) -- (7, 1) -- (8, 0) -- (9, 1) -- (10, 2) -- (11, 3) -- (12, 2) -- (13, 1) -- (14, 0) -- (15, 1) -- (16, 2) -- (17, 3) -- (18, 2) -- (19, 1) -- (20, 0);

        \draw[ultra thick, red] (4, 4) -- (4, 0);
        \draw[ultra thick, red] (11, 4) -- (11, 0);
        \draw[ultra thick, red] (17, 4) -- (17, 0);
    \end{tikzpicture}
    \caption{The \textcolor{red}{vertical lines} decompose the Dyck path into its bounce factors.}
    \label{fig:bouncefactors}
\end{figure}

\begin{proof}[Proof of \Cref{md}]
    We will show that $\des (d, \bperm d) = 2n - 1 - \bpk d$. 
    The key observations that will help us prove the result are the following.
    \begin{enumerate}
        \item At any time during \Cref{algo}, if the labels on a subsequence of steps of $d$ form a decreasing sequence, it will remain decreasing throughout the rest of the algorithm. \label{obs1}

        \item At any time during \Cref{algo}, if there exist two down-steps that only have unlabeled up-steps between them and with the first down-step having a larger label than the second, then the part of the Dyck path between these two down-steps (both inclusive) will get a decreasing sequence of labels. \label{obs2}
    \end{enumerate}
    
    We will use these observations to show that when the algorithm ends, in each bounce factor, the labels form a decreasing sequence. 
    This will show that $\des(d, \bperm d) \geq 2n - 1 - \bpk d$, which along with \Cref{upperbound} proves the result.

    By the definition of the algorithm, the first bounce factor will initially be labeled with a decreasing sequence and hence, by observation \ref{obs1} will remain decreasing. 
    Also, once we have labeled all the (up) steps in the first factor, by the rules of the algorithm, we also label all the down-steps in the second factor. 
    Hence, the labels on the down-steps in the second factor form a decreasing sequence.
    
    We now describe why the labels on the \emph{entire} second factor form a decreasing sequence. 
    If all the down-steps appear before all the up-steps in the second factor, this follows from the rules of the algorithm. 
    If not, this follows by repeatedly applying observation \ref{obs2} and noting that the down-steps in the second factor have decreasing labels. 
    The same reasoning can be extended to show that the remaining factors all end up with labels forming decreasing sequences.
\end{proof}

\begin{remark}\label{bpk=ht}
    The distribution of the statistic `number of peaks of the bounce path' on Dyck paths (i.e., $d \mapsto \bpk d$) is equidistributed with the classical `height' statistic. 
    One way to prove this is to note that the Dyck paths $d$ with $\bcomp d = (c_1, \ldots, c_k)$ are in bijection with Dyck paths that have exactly $c_i$ up-steps at height $i$. 
    Hence, the distribution of the statistic `maximum number of descents in a canon labeling' (i.e., $d \mapsto m_d$) is given by \cite[\href{https://oeis.org/A80936}{A80936}]{oeis}.
\end{remark}

\begin{example}\label{eg: non-unique maximazers}
    For $d = UUDUDDUD$, we have $C_d(t) = 4t^2 + 16t^3 + 4t^4$ (see \Cref{fig:symdycks}). 
    Since this Dyck path has semilength $4$ and $\bpk d = 3$, we get that $m_d = 2\cdot 4 - 1 - 3 = 4$, as suggested by $C_d(t)$. 
    Note that $[t^4]C_d(t) \neq 1$, that is, there are permutations other than $\bperm d$ that maximize the number of descents. 
    For example, in this case $\bperm d = 4132$, and we have $\des(d, \sigma) = 4$ for all $\sigma \in \{3421, 4132, 4231, 4321\}$.
\end{example}

In \Cref{sec:valleys}, we will see another method to generate a permutation $\sigma$ such that $\des(d, \sigma) = m_d$. 
This will involve swapping valleys to change from $\bounce d$ to $d$. 
We then extend this method in \Cref{sec: interpret leader} to generate more such permutations, and eventually describe a method to generate \emph{all} permutations that maximize the number of descents. 
This involves a generalized notion of bounce paths and linear extensions of posets.

\section{More distributional properties}\label{sec: more distributional properties}

In this section, we collect some consequences of the results in the previous section, and use them to get new insights on $C_d$. 
The first immediate corollary is the minimum number of descents that a canon permutation associated to a Dyck path can have.

\begin{corollary}\label{mindes}
    For any $d \in \Dyck_n$, we have
    \begin{equation*}
        \min\{\des(d, \sigma) \mid \sigma \in \s_n\} = \bpk d - \lpk d.
    \end{equation*}
\end{corollary}

\begin{proof}
    This is an immediate consequence of \Cref{sym,md}.
\end{proof}
It immediately follows that the constant term of $C_d$ is $0$, unless $d$ is the Dyck path $(UD)^n$ (which realizes the unique non-decreasing canon permutation $1122 \cdots nn$).

\begin{definition}
    For a Dyck path $d$ of semilength $n$, we set
    \begin{equation*}
        M_d \coloneqq \{\sigma \in \s_n \mid \des(d, \sigma) = m_d\}.
    \end{equation*}
\end{definition}
Hence, $M_d$ is the set of permutations that can be used to label $d$ to obtain canon permutations with the maximum possible number of descents. 
\Cref{algo} chooses a permutation in $M_d$, which we denote $\bperm d$, for each Dyck path $d$. 

As mentioned before, the decreasing permutation need not always be in $M_d$. 
We now characterize the Dyck paths for which this is the case. 
We use $\delta_n$ to denote the decreasing permutation $n(n-1)\cdots 1$ in $\s_n$.

\begin{corollary}\label{cor: descreasing is maximizer}
    For $d \in \Dyck_n$, we have $\delta_n \in M_d$ if and only if $\pk d = \bpk d$. 
    These Dyck paths are counted by \cite[\href{https://oeis.org/A1519}{A1519}]{oeis}.
\end{corollary}

\begin{proof}
    It is not difficult to check that the non-descents in $\can(d, \delta_n)$ correspond to the peaks of $d$. 
    Hence, $\des(d, \delta_n) = 2n - 1 - \pk d$. 
    Comparing this with the expression for $m_d$ from \Cref{md}, we get that $\delta_n \in M_d$ if and only if $\pk d = \bpk d$.

    To show that these Dyck paths are counted by the given OEIS sequence, we show that their generating function is
    \begin{equation*}
        \frac{1 - 2x}{1 - 3x + x^2}.
    \end{equation*}
    To do this, we count them based on their bounce paths. 
    Recall that for a Dyck path $d$, we use $\bcomp d$ to denote the heights of the peaks of its bounce path. 
    It is well-known that the \emph{total} number of Dyck paths $d$ with $\bcomp d = (c_1, c_2, \ldots, c_k)$ is
    \begin{equation*}
        \prod_{i = 1}^{k - 1} \binom{c_i + c_{i + 1} - 1}{c_{i + 1}}.
    \end{equation*}
    This is because, given the composition, we only have to choose lattice walks in certain rectangles between peaks of the bounce path to obtain a Dyck path with the corresponding bounce path (see \Cref{fig:dyckfrombounce}).

    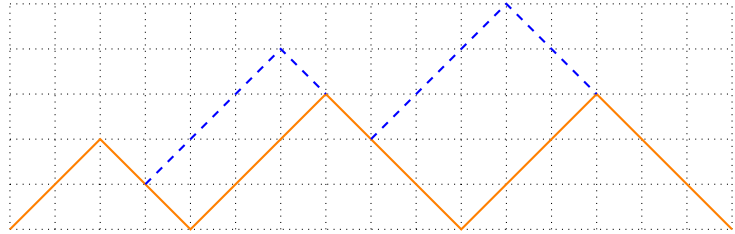
\begin{figure}[H]
        \centering
        \begin{tikzpicture}[scale = 0.6]
            \draw[dotted] (0, 0) grid (16, 5);

            \draw[thick, orange] (0, 0) -- (2, 2) -- (4, 0) -- (7, 3) -- (10, 0) -- (13, 3) -- (16, 0);

            \draw[thick, dashed, blue] (3, 1) -- (6, 4) -- (7, 3);
            \draw[thick, dashed, blue] (8, 2) -- (11, 5) -- (13, 3);            
        \end{tikzpicture}
        \caption{Dyck paths with the given bounce path are those that lie between the \textcolor{orange}{bounce path} and the \textcolor{blue}{boundary above it}.}
        \label{fig:dyckfrombounce}
    \end{figure}

    To make sure that the number of peaks of $d$ matches the number of peaks of $\bounce d$, we require that all the up-steps in each bounce factor of $d$ form a consecutive sequence of steps. 
    Hence, the number of Dyck paths $d$ of semilength $n$ where $\pk d = \bpk d$ is given by
    \begin{equation*}
        \sum_{c = (c_1, c_2, \ldots, c_k)} c_1c_2 \cdots c_{k - 1}
    \end{equation*}
    where the sum is over all compositions $c$ of $n$ (see \Cref{ex:pk=bpk}). 
    It can be checked that these numbers have the required generating function.
\end{proof}

\begin{example}\label{ex:pk=bpk}
    The $2 \cdot 3 = 6$ Dyck paths $d$ with $\bcomp d = (2, 3, 3)$ and $\delta_8 \in M_d$ are shown in \Cref{fig:pk=bpk}.
\end{example}

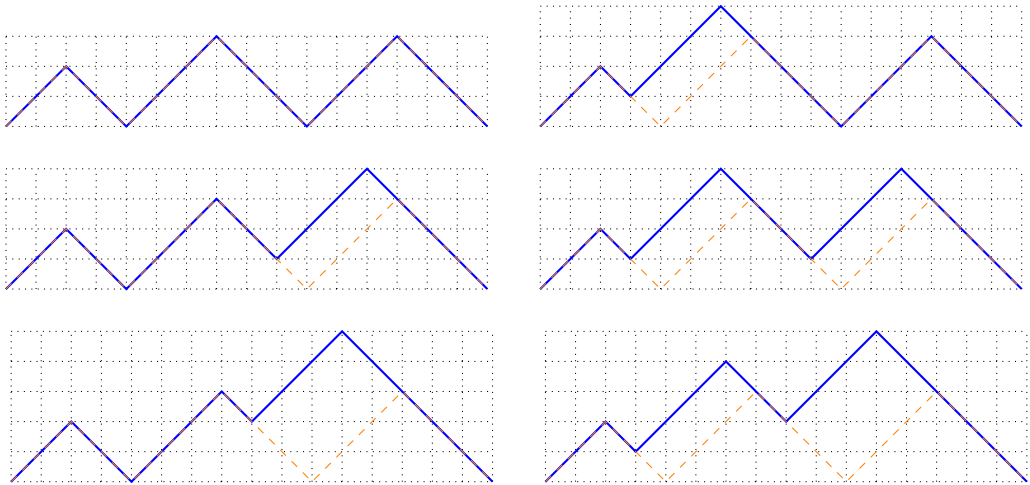
\begin{figure}[H]
    \centering
    \begin{tikzpicture}[scale = 0.4]
        \draw[thin, dotted] (0, 0) grid (16, 3);

        \draw[thick, blue] (0, 0) -- (2, 2) -- (4, 0) -- (7, 3) -- (10, 0) -- (13, 3) -- (16, 0);

        \draw[orange, dashed] (0, 0) -- (2, 2) -- (4, 0) -- (7, 3) -- (10, 0) -- (13, 3) -- (16, 0);
    \end{tikzpicture}
    \hspace{0.4cm}
    \begin{tikzpicture}[scale = 0.4]
        \draw[thin, dotted] (0, 0) grid (16, 4);

        \draw[thick, blue] (0, 0) -- (2, 2) -- (3, 1) -- (6, 4) -- (10, 0) -- (13, 3) -- (16, 0);

        \draw[orange, dashed] (0, 0) -- (2, 2) -- (4, 0) -- (7, 3) -- (10, 0) -- (13, 3) -- (16, 0);
    \end{tikzpicture}
    \vspace{0.5cm}

    \begin{tikzpicture}[scale = 0.4]
        \draw[thin, dotted] (0, 0) grid (16, 4);

        \draw[thick, blue] (0, 0) -- (2, 2) -- (4, 0) -- (7, 3) -- (9, 1) -- (12, 4) -- (16, 0);

        \draw[orange, dashed] (0, 0) -- (2, 2) -- (4, 0) -- (7, 3) -- (10, 0) -- (13, 3) -- (16, 0);
    \end{tikzpicture}
    \hspace{0.4cm}
    \begin{tikzpicture}[scale = 0.4]
        \draw[thin, dotted] (0, 0) grid (16, 4);

        \draw[thick, blue] (0, 0) -- (2, 2) -- (3, 1) -- (6, 4) -- (9, 1) -- (12, 4) -- (16, 0);

        \draw[orange, dashed] (0, 0) -- (2, 2) -- (4, 0) -- (7, 3) -- (10, 0) -- (13, 3) -- (16, 0);
    \end{tikzpicture}
    \vspace{0.5cm}

    \begin{tikzpicture}[scale = 0.4]
        \draw[thin, dotted] (0, 0) grid (16, 5);

        \draw[thick, blue] (0, 0) -- (2, 2) -- (4, 0) -- (7, 3) -- (8, 2) -- (11, 5) -- (16, 0);

        \draw[orange, dashed] (0, 0) -- (2, 2) -- (4, 0) -- (7, 3) -- (10, 0) -- (13, 3) -- (16, 0);
    \end{tikzpicture}
    \hspace{0.4cm}
    \begin{tikzpicture}[scale = 0.4]
        \draw[thin, dotted] (0, 0) grid (16, 5);

        \draw[thick, blue] (0, 0) -- (2, 2) -- (3, 1) -- (6, 4) -- (8, 2) -- (11, 5) -- (16, 0);

        \draw[orange, dashed] (0, 0) -- (2, 2) -- (4, 0) -- (7, 3) -- (10, 0) -- (13, 3) -- (16, 0);
    \end{tikzpicture}
    \caption{Dyck paths $d$ with $\bcomp d = (2, 3, 3)$ and $\pk d = \bpk d$.}
    \label{fig:pk=bpk}
\end{figure}

\begin{remark}
    There are several other classes of Dyck paths counted by \cite[\href{https://oeis.org/A1519}{A1519}]{oeis}. 
    For example, Dyck paths whose height is at most $3$ (graded by semilength). 
    Another class is Dyck paths that are \emph{nondecreasing}, i.e., the heights of their valleys form a nondecreasing sequence. 
    It would be a fun exercise to find bijections between the class of Dyck paths described in the corollary above and the other classes mentioned in \cite[\href{https://oeis.org/A1519}{A1519}]{oeis}.
\end{remark}

For the class of bounce paths, there is a unique permutation realizing the maximum number of descents. This role is played by the permutation $\delta_n$, and the following stronger statement holds.

\begin{corollary}\label{cor: unique decreasing maximizer}
    For $d \in \Dyck_n$, we have $M_d = \{\delta_n\}$ if and only if $d = \bounce d$.
\end{corollary}

\begin{proof}
    Let $d = d_1d_2 \cdots d_k$ be the decomposition of $d$ into primitive factors. 
    When maximizing the number of descents, the last step of $d_i$ must get a larger label than the first step of $d_{i + 1}$ for all $i \in [k - 1]$.
    
    Suppose that $d = \bounce d$. 
    If $d = U^nD^n$, then one can check that $M_d = \{\delta_n\}$. 
    In the case $d = U^{c_1}D^{c_1}U^{c_2}D^{c_2} \cdots U^{c_k}D^{c_k}$, the fact that $M_d = \{\delta_n\}$ follows from the case when $k = 1$ and the observation about primitive factors made above.

    Next, suppose that $d \neq \bounce d$. 
    We claim that, in this case, $\bperm d \neq \delta_n$, which will show that $M_d \neq \{\delta_n\}$. 
    We first assume that $d$ is primitive. 
    Note that when applying \Cref{algo}, since $d \neq \bounce d$ and $d$ is primitive, at some point in the labeling, the first string up-steps will get non-consecutive labels (see \Cref{fig:nonconsec}).

    \begin{figure}[H]
        \centering
        \begin{tikzpicture}[scale = 0.9]
            \draw[dotted] (0, 0) grid (7, 4);
            \draw[dotted] (7, 0) -- (7.5, 0);
            \draw[dotted] (7, 1) -- (7.5, 1);
            \draw[dotted] (7, 2) -- (7.5, 2);
            \draw[dotted] (7, 3) -- (7.5, 3);
            \draw[dotted] (7, 4) -- (7.5, 4);
            \draw[thick, blue] (0, 0) -- (1, 1) -- (2, 2) -- (3, 3) -- (4, 4) -- (5, 3) -- (6, 2) -- (7, 3);
            \node at (0 + 0.4, 0 + 0.75) {\scriptsize 8};
            \node at (1 + 0.4, 1 + 0.75) {\scriptsize 7};
            \node[circle, draw = red, inner sep = 1pt] at (2 + 0.4, 2 + 0.75) {\scriptsize 6};
            \node[circle, draw = red, inner sep = 1pt] at (2 + 1 + 0.4, 2 + 1 + 0.75) {\scriptsize 5};
            \node at (3 + 1 + 0.6, 3 + 1 - 0.25) {\scriptsize 8};
            \node at (4 + 1 + 0.6, 2 + 1 - 0.25) {\scriptsize 7};
            \node at (5 + 1 + 0.4, 1 + 1 + 0.75) {\scriptsize \color{red}\textbf{6}};

            \node at (8.75, 2) {$\rightarrow$};
        \end{tikzpicture}
        \hfill
        \begin{tikzpicture}[scale = 0.9]
            \draw[dotted] (0, 0) grid (7, 4);
            \draw[dotted] (7, 0) -- (7.5, 0);
            \draw[dotted] (7, 1) -- (7.5, 1);
            \draw[dotted] (7, 2) -- (7.5, 2);
            \draw[dotted] (7, 3) -- (7.5, 3);
            \draw[dotted] (7, 4) -- (7.5, 4);
            \draw[thick, blue] (0, 0) -- (1, 1) -- (2, 2) -- (3, 3) -- (4, 4) -- (5, 3) -- (6, 2) -- (7, 3);
            \node at (0 + 0.4, 0 + 0.75) {\scriptsize 8};
            \node at (1 + 0.4, 1 + 0.75) {\scriptsize 7};
            \node at (2 + 0.4, 2 + 0.75) {\scriptsize 5};
            \node at (2 + 1 + 0.4, 2 + 1 + 0.75) {\scriptsize 4};
            \node at (3 + 1 + 0.6, 3 + 1 - 0.25) {\scriptsize 8};
            \node at (4 + 1 + 0.6, 2 + 1 - 0.25) {\scriptsize 7};
            \node at (5 + 1 + 0.4, 1 + 1 + 0.75) {\scriptsize 6};
        \end{tikzpicture}
        \caption{A point in the algorithm where the first sequence of up-steps get non-consecutive labels.}
        \label{fig:nonconsec}
    \end{figure}
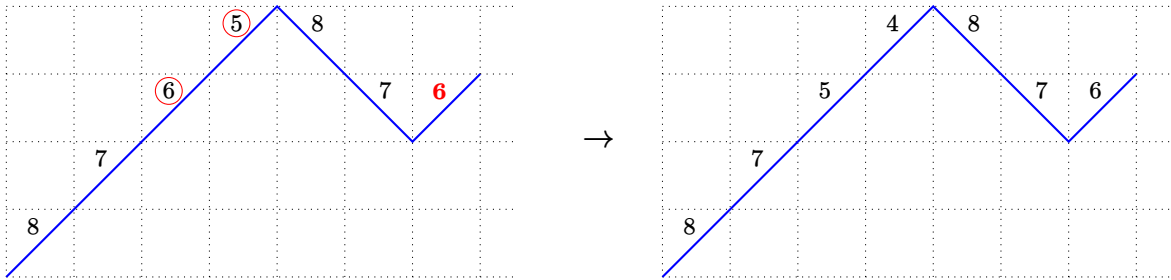
    
    By the nature of the relabeling process, this remains true till the end of the algorithm. 
    Hence, we cannot have $\bperm d = \delta_n$. 
    A similar logic can be applied when $d$ is not primitive by focusing on the first primitive factor that has more than one peak.
\end{proof}


We now show that the polynomials $C_d(t)$ have no \emph{internal zeroes}. 
A polynomial $f(t)$ is said to have an \emph{internal zero} if there exist $0 \leq i < j < k$ such that $[t^i]f(t), [t^k]f(t) \neq 0$ but $[t^j]f(t) = 0$.

\begin{proposition}\label{niz}
    For any Dyck path $d$, the polynomial $C_d(t)$ has no internal zeroes.
\end{proposition}

\begin{proof}
    We prove this by induction on the semilength $n$ of $d$. 
    We can assume that $d$ is primitive. 
    This can be derived from the fact that if $d = d_1d_2$, then we can always find $\sigma \in \s_n$ such that $\des(d, \sigma) = \des(d_1, \sigma_1) + \des(d_2, \sigma_2)$ for any labeling $\sigma_1$ and $\sigma_2$ of $d_1$ and $d_2$ respectively.

    Since $d$ is primitive, by \Cref{mindes}, the minimum power of $t$ in $C_d(t)$ is $\bpk d$. 
    Let $d'$ be the Dyck path obtained by deleting the first up-step of $d$ and its corresponding down-step. 
    Note that this gives us $\bpk d' \in \{\bpk d, \bpk d - 1\}$. 
    We will show that if there is a permutation $\sigma' \in \s_{n - 1}$ such that $\des(d', \sigma') = k$, then there is a permutation $\sigma \in \s_n$ such that $\des(d, \sigma) = k + 1$. 
    This will prove the result since, by induction, we can always find such a $\sigma'$ for any $k$ that satisfies $\bpk d \leq k \leq 2n - 3 - \bpk d$.

    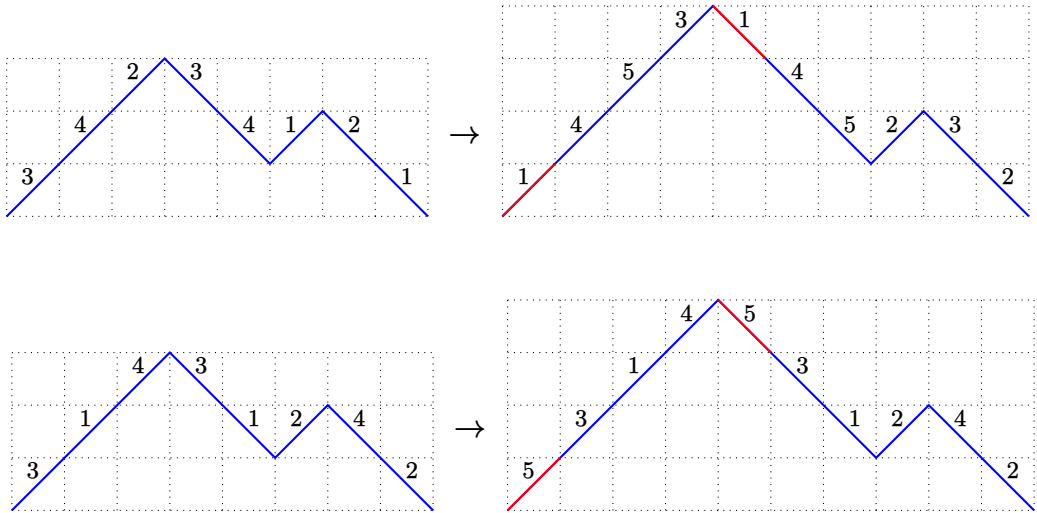
\begin{figure}[H]
        \centering
        \begin{tikzpicture}[scale = 0.7]
            \draw[dotted] (0, 0) grid (8, 3);
            \draw[thick, blue] (0, 0) -- (3, 3) -- (5, 1) -- (6, 2) -- (8, 0);
            \node at (0 + 0.4, 0 + 0.75) {\scriptsize 3};
            \node at (1 + 0.4, 1 + 0.75) {\scriptsize 4};
            \node at (2 + 0.4, 2 + 0.75) {\scriptsize 2};
            \node at (3 + 0.6, 3 - 0.25) {\scriptsize 3};
            \node at (4 + 0.6, 2 - 0.25) {\scriptsize 4};
            \node at (5 + 0.4, 1 + 0.75) {\scriptsize 1};
            \node at (6 + 0.6, 2 - 0.25) {\scriptsize 2};
            \node at (7 + 0.6, 1 - 0.25) {\scriptsize 1};
            
            \node at (8.7, 1.5) {$\rightarrow$};
        \end{tikzpicture}
        \begin{tikzpicture}[scale = 0.7]
            \draw[dotted] (0, 0) grid (10, 4);
            \draw[thick, blue] (0, 0) -- (4, 4) -- (7, 1) -- (8, 2) -- (10, 0);
            \draw[thick, red] (0, 0) -- (1, 1);
            \draw[thick, red] (4, 4) -- (5, 3);
            \node at (0 + 0.4, 0 + 0.75) {\scriptsize 1};
            \node at (1 + 0.4, 1 + 0.75) {\scriptsize 4};
            \node at (2 + 0.4, 2 + 0.75) {\scriptsize 5};
            \node at (3 + 0.4, 3 + 0.75) {\scriptsize 3};
            \node at (4 + 0.6, 4 - 0.25) {\scriptsize 1};
            \node at (5 + 0.6, 3 - 0.25) {\scriptsize 4};
            \node at (6 + 0.6, 2 - 0.25) {\scriptsize 5};
            \node at (7 + 0.4, 1 + 0.75) {\scriptsize 2};
            \node at (8 + 0.6, 2 - 0.25) {\scriptsize 3};
            \node at (9 + 0.6, 1 - 0.25) {\scriptsize 2};
        \end{tikzpicture}
        \vspace{1cm}
        
        \begin{tikzpicture}[scale = 0.7]
            \draw[dotted] (0, 0) grid (8, 3);
            \draw[thick, blue] (0, 0) -- (3, 3) -- (5, 1) -- (6, 2) -- (8, 0);
            \node at (0 + 0.4, 0 + 0.75) {\scriptsize 3};
            \node at (1 + 0.4, 1 + 0.75) {\scriptsize 1};
            \node at (2 + 0.4, 2 + 0.75) {\scriptsize 4};
            \node at (3 + 0.6, 3 - 0.25) {\scriptsize 3};
            \node at (4 + 0.6, 2 - 0.25) {\scriptsize 1};
            \node at (5 + 0.4, 1 + 0.75) {\scriptsize 2};
            \node at (6 + 0.6, 2 - 0.25) {\scriptsize 4};
            \node at (7 + 0.6, 1 - 0.25) {\scriptsize 2};
            
            \node at (8.7, 1.5) {$\rightarrow$};
        \end{tikzpicture}
        \begin{tikzpicture}[scale = 0.7]
            \draw[dotted] (0, 0) grid (10, 4);
            \draw[thick, blue] (0, 0) -- (4, 4) -- (7, 1) -- (8, 2) -- (10, 0);
            \draw[thick, red] (0, 0) -- (1, 1);
            \draw[thick, red] (4, 4) -- (5, 3);
            \node at (0 + 0.4, 0 + 0.75) {\scriptsize 5};
            \node at (1 + 0.4, 1 + 0.75) {\scriptsize 3};
            \node at (2 + 0.4, 2 + 0.75) {\scriptsize 1};
            \node at (3 + 0.4, 3 + 0.75) {\scriptsize 4};
            \node at (4 + 0.6, 4 - 0.25) {\scriptsize 5};
            \node at (5 + 0.6, 3 - 0.25) {\scriptsize 3};
            \node at (6 + 0.6, 2 - 0.25) {\scriptsize 1};
            \node at (7 + 0.4, 1 + 0.75) {\scriptsize 2};
            \node at (8 + 0.6, 2 - 0.25) {\scriptsize 4};
            \node at (9 + 0.6, 1 - 0.25) {\scriptsize 2};
        \end{tikzpicture}
        \caption{Instances of the map used in the proof of \Cref{niz}.}
        \label{fig:niz}
    \end{figure}

    Let $\sigma' \in \s_n$ be such that $\des(d', \sigma') = k$. 
    Suppose that $d$ starts with $m$ up-steps followed by a down-step. 
    Let the first $m$ terms of $\sigma'$ be $a_1, a_2, \ldots, a_m$. 
    Suppose that the first $m + 1$ terms of $\can(d', \sigma')$ are $a_1, a_2, \ldots, a_m, a_{m + 1}$ (note that we might have $a_{m + 1} = a_1$). 
    If $a_m < a_{m + 1}$, then we set $\sigma \in \s_n$ to be the permutation obtained by increasing each term of $\sigma'$ by one and appending $1$ at the beginning (see \Cref{fig:niz}). 
    Note that the first $m + 3$ terms of $\can(d, \sigma)$ are $1, a_1 + 1, a_2 + 1, \ldots, a_m + 1, 1, a_{m + 1} + 1$. 
    It can be verified that $\des(d, \sigma) = k + 1$.

    If $a_m > a_{m + 1}$, we set $\sigma \in \s_n$ to be the permutation obtained by appending $n$ at the beginning of $\sigma'$. 
    It can be similarly be verified that this permutation satisfies our requirements.
\end{proof}

\begin{remark}
    For any $d \in \Dyck_n$, consider the polynomial
    \begin{equation*}
        \tilde{C_d}(t) \coloneqq \sum_{\substack{\sigma \in \s_n\\ \sigma_1 \in \{1, n\}}} t^{\des(d, \sigma)}.
    \end{equation*}
    The proof of the above proposition shows that $[t^i]C_d(t) \neq 0$ if and only if $[t^i]\tilde{C_d}(t) \neq 0$ for any $i \geq 0$. 
    Along with the proof of \Cref{sym}, we get that $C_d$ and $\tilde{C_d}$ have the same degree $m_d$ and are both palindromic with respect to the same degree $2n - 1 - \lpk d$.
\end{remark}

\section{Valley swapping for canon permutations}\label{sec:valleys}

In Section \ref{sec:symdeg}, we have seen one method to choose a permutation from $M_d$, namely $\bperm d$, for any Dyck path $d$. 
In this section, we describe a different way to construct a permutation in $M_d$. 
With this method, we can walk around Dyck paths with the same bounce path, labeled to have the maximum number of descents (and fixed descent positions). 
Starting from a bounce path, we can reach any Dyck path with this underlying bounce path by swapping valleys to peaks. 
To preserve the descent structure in the corresponding canon permutation, we need to relabel the steps of the Dyck paths after each valley swap. 
The following definition introduces such a procedure.

\begin{definition}\label{def: valley swap}
    Let $\can(d, \sigma)$ be a canon permutation represented as a labeled Dyck path. 
    For a valley in the Dyck path, performing the \emph{valley swap} corresponding to that valley results in the canon permutation $\can(d', \sigma')$ where
    \begin{itemize}
        \item $d'$ is the Dyck path obtained by swapping the up and down-steps of the chosen valley, thus turning it into a peak, and
        \item if the valley in $\can(d, \sigma)$ has labels $i, j$ where $i < j$, then $\sigma'$ is obtained from $\sigma$ by replacing $i$ with $j$ and $k$ with $k - 1$ for all $k \in [i + 1, j]$.
    \end{itemize}
\end{definition}

Although we haven't specified this in the definition, we will only be dealing with the case when the valley being swapped has the larger label $j$ on its down-step. 
Note that after the valley swap is applied, the new peak has labels $j$ and $j - 1$.

\begin{example}
    In \Cref{fig:vswapex}, we swap the valley with labels $5$ and $2$. 
    The new peak will obtain labels $5$ and $4$. 
    The circled labels are the ones that change once the swap is done. 
    The circled $2$ is replaced with $5$ and all other circled labels are reduced by one.
\end{example}

\begin{figure}[h]
    \centering
    \begin{tikzpicture}[scale = 0.7]
        \draw[dotted] (0, 0) grid (16, 4);
        \draw[thick, blue] (0, 0) -- (1, 1) -- (2, 2) -- (3, 3) -- (4, 2) -- (5, 3) -- (6, 4) -- (7, 3) -- (8, 2) -- (9, 1) -- (10, 2) -- (11, 1) -- (12, 2) -- (13, 1) -- (14, 0) -- (15, 1) -- (16, 0);
        \draw[ultra thick, dotted, red] (10, 2) -- (11, 3) -- (12, 2);
        \node at (0 + 0.4, 0 + 0.75) {\scriptsize 8};
        \node[circle, draw = red, inner sep = 1pt] at (1 + 0.4, 1 + 0.75) {\scriptsize 4};
        \node[circle, draw = red, inner sep = 1pt] at (2 + 0.4, 2 + 0.75) {\scriptsize 3};
        \node at (3 + 0.6, 3 - 0.25) {\scriptsize 8};
        \node at (4 + 0.4, 2 + 0.75) {\scriptsize 7};
        \node[circle, draw = red, inner sep = 1pt] at (5 + 0.4, 3 + 0.75) {\scriptsize 5};
        \node[circle, draw = red, inner sep = 1pt] at (6 + 0.6, 4 - 0.25) {\scriptsize 4};
        \node[circle, draw = red, inner sep = 1pt] at (7 + 0.6, 3 - 0.25) {\scriptsize 3};
        \node at (8 + 0.6, 2 - 0.25) {\scriptsize 7};
        \node at (9 + 0.4, 1 + 0.75) {\scriptsize 6};
        \node at (10 + 0.6, 2 - 0.25) {\scriptsize 5};
        \node at (11 + 0.4, 1 + 0.75) {\scriptsize 2};
        \node at (12 + 0.6, 2 - 0.25) {\scriptsize 6};
        \node[circle, draw = red, inner sep = 1pt] at (13 + 0.6, 1 - 0.25) {\scriptsize 2};
        \node at (14 + 0.4, 0 + 0.75) {\scriptsize 1};
        \node at (15 + 0.6, 1 - 0.25) {\scriptsize 1};

        \node at (10 + 0.4, 2 + 0.75) {\scriptsize \color{red} 5};
        \node at (11 + 0.6, 3 - 0.25) {\scriptsize \color{red} 4};
    \end{tikzpicture}\\[0.2cm]
    $\downarrow$\\[0.4cm]
    \begin{tikzpicture}[scale = 0.7]
        \draw[dotted] (0, 0) grid (16, 4);
        \draw[thick, blue] (0, 0) -- (1, 1) -- (2, 2) -- (3, 3) -- (4, 2) -- (5, 3) -- (6, 4) -- (7, 3) -- (8, 2) -- (9, 1) -- (10, 2) -- (11, 3) -- (12, 2) -- (13, 1) -- (14, 0) -- (15, 1) -- (16, 0);
        \node at (0 + 0.4, 0 + 0.75) {\scriptsize 8};
        \node at (1 + 0.4, 1 + 0.75) {\scriptsize 3};
        \node at (2 + 0.4, 2 + 0.75) {\scriptsize 2};
        \node at (3 + 0.6, 3 - 0.25) {\scriptsize 8};
        \node at (4 + 0.4, 2 + 0.75) {\scriptsize 7};
        \node at (5 + 0.4, 3 + 0.75) {\scriptsize 4};
        \node at (6 + 0.6, 4 - 0.25) {\scriptsize 3};
        \node at (7 + 0.6, 3 - 0.25) {\scriptsize 2};
        \node at (8 + 0.6, 2 - 0.25) {\scriptsize 7};
        \node at (9 + 0.4, 1 + 0.75) {\scriptsize 6};
        \node at (10 + 0.4, 2 + 0.75) {\scriptsize 5};
        \node at (11 + 0.6, 3 - 0.25) {\scriptsize 4};
        \node at (12 + 0.6, 2 - 0.25) {\scriptsize 6};
        \node at (13 + 0.6, 1 - 0.25) {\scriptsize 5};
        \node at (14 + 0.4, 0 + 0.75) {\scriptsize 1};
        \node at (15 + 0.6, 1 - 0.25) {\scriptsize 1};
    \end{tikzpicture}
    \caption{Example of a valley swap and relabeling.}
    \label{fig:vswapex}
\end{figure}
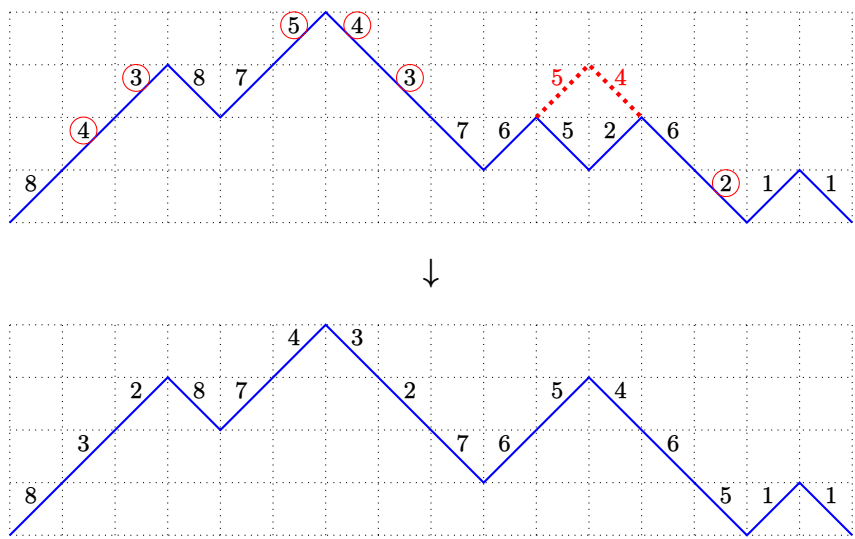

Before describing the algorithm, we need to recall a well-known poset defined on $\Dyck_n$. 
For $d_1, d_2 \in \Dyck_n$, we say that $d_1$ is \emph{under} $d_2$ if $d_1$ lies below $d_2$. 
The cover relations in this poset corresponds to changing valleys to peaks. 
This poset is in fact a lattice and is referred to as the Dyck lattice or the Stanley lattice \cite{stanleyec2}. 
For example, $d_1 = UUDUDDUD$ is under $d_2 = UUUDDUDD$. 
Note that $\bounce d$ is under $d$ for any Dyck path $d$.

\begin{algorithm}\label{valgo}
    Let $d$ be a Dyck path of semilength $n$ and $c = \can(\bounce d, \delta_n)$. 
    Repeat the following until $c$ has $d$ as its underlying Dyck path: 
    Find the first (leftmost) valley in $c$ that can be swapped such that the resulting Dyck path is under $d$. 
    Redefine $c$ to be the canon permutation obtained after performing the valley swap corresponding to this valley (according to Definition \ref{def: valley swap}).
\end{algorithm}

The essence of the algorithm is to move up from $\bounce d$ to $d$ in the Dyck lattice. 
For $\bounce d$, we know that the only permutation that maximizes the number of descents is the decreasing permutation (Corollary \ref{cor: unique decreasing maximizer}). 
Since we also know that $m_d = m_{\bounce d}$, the idea is to maintain the number of descents as we move up in the Dyck lattice. 
We will show that the valley swaps in Algorithm \ref{valgo} do this. In fact, the descent positions are preserved.

The way we draw the algorithm for $d$ is as follows. 
We start with the \textcolor{orange}{`changing path'}; this is initially $\bounce d$ labeled with the decreasing permutation. 
We draw the steps of $d$ that aren't in the \textcolor{orange}{lower} path \textcolor{blue!50}{on top}. 
Hence, at each stage we perform the valley swap to the first (leftmost) valley in the \textcolor{orange}{lower} path for which the resulting Dyck path lies under the blue path (see \Cref{fig:valgoex}).

\begin{definition}
    For any Dyck path $d$, we set $\vperm d$ to be the permutation such that $\can(d, \vperm d)$ is the canon permutation obtained after applying \Cref{valgo} to $d$.
\end{definition}

\begin{example}
    We exhibit a few steps of the algorithm for $d = U^3(DU)^2D^2U^2D^2UD^2$ in \Cref{fig:valgoex}. 
    Hence, $\vperm d = 8 6 1 7 2 5 3 4$.
\end{example}

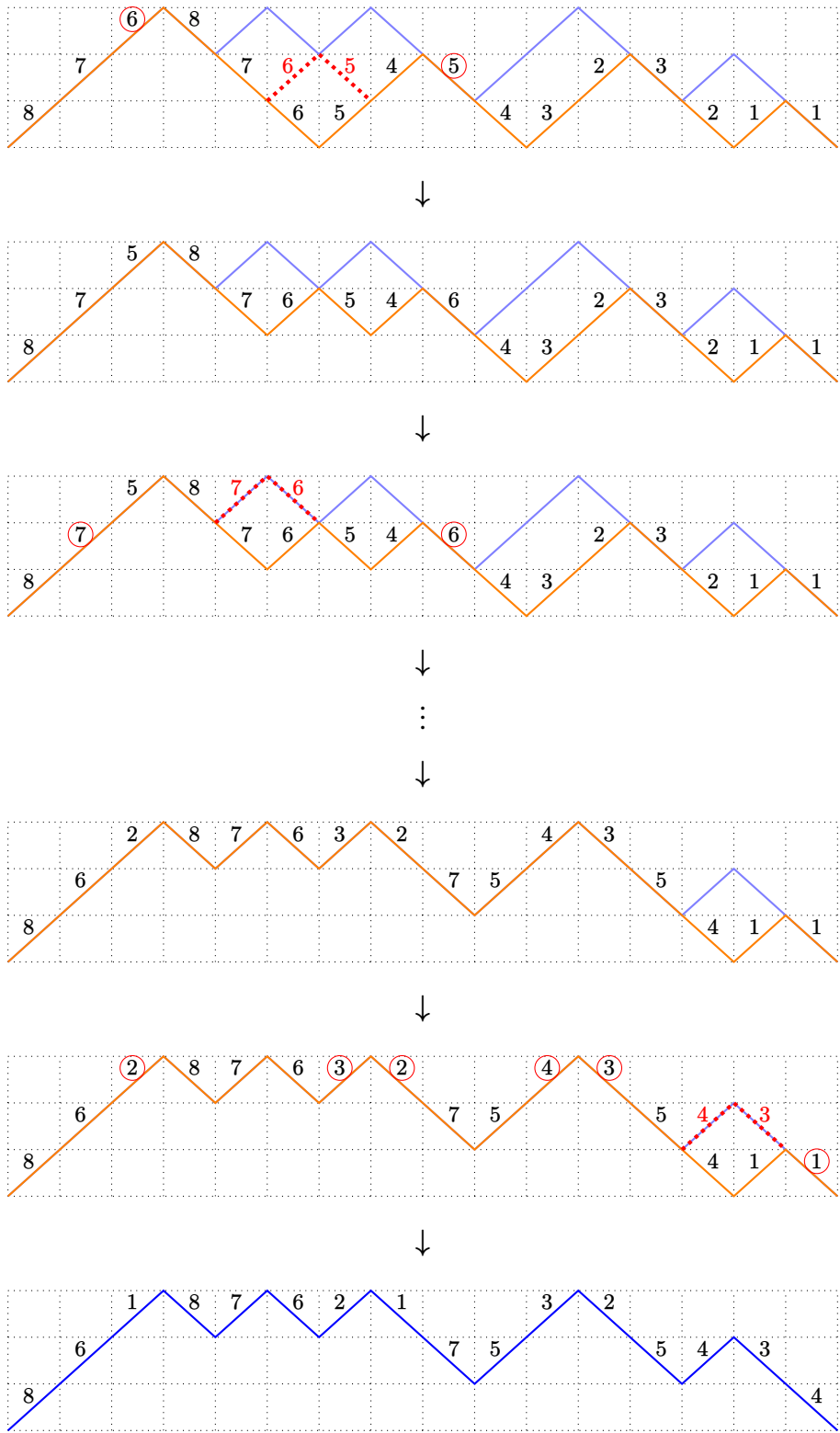
\begin{figure}
    \centering
    \begin{tikzpicture}[scale = 0.75, yscale = 0.9]
        \draw[dotted] (0, 0) grid (16, 3);
        \draw[thick, blue!50] (0, 0) -- (1, 1) -- (2, 2) -- (3, 3) -- (4, 2) -- (5, 3) -- (6, 2) -- (7, 3) -- (8, 2) -- (9, 1) -- (10, 2) -- (11, 3) -- (12, 2) -- (13, 1) -- (14, 2) -- (15, 1) -- (16, 0);
        \draw[thick, orange] (0, 0) -- (1, 1) -- (2, 2) -- (3, 3) -- (4, 2) -- (5, 1) -- (6, 0) -- (7, 1) -- (8, 2) -- (9, 1) -- (10, 0) -- (11, 1) -- (12, 2) -- (13, 1) -- (14, 0) -- (15, 1) -- (16, 0);
        \node at (0 + 0.4, 0 + 0.75) {\scriptsize 8};
        \node at (1 + 0.4, 1 + 0.75) {\scriptsize 7};
        \node[circle, draw = red, inner sep = 1pt] at (2 + 0.4, 2 + 0.75) {\scriptsize 6};
        \node at (3 + 0.6, 3 - 0.25) {\scriptsize 8};
        \node at (4 + 0.6, 2 - 0.25) {\scriptsize 7};
        \node at (5 + 0.6, 1 - 0.25) {\scriptsize 6};
        \node at (6 + 0.4, 0 + 0.75) {\scriptsize 5};
        \node at (7 + 0.4, 1 + 0.75) {\scriptsize 4};
        \node[circle, draw = red, inner sep = 1pt] at (8 + 0.6, 2 - 0.25) {\scriptsize 5};
        \node at (9 + 0.6, 1 - 0.25) {\scriptsize 4};
        \node at (10 + 0.4, 0 + 0.75) {\scriptsize 3};
        \node at (11 + 0.4, 1 + 0.75) {\scriptsize 2};
        \node at (12 + 0.6, 2 - 0.25) {\scriptsize 3};
        \node at (13 + 0.6, 1 - 0.25) {\scriptsize 2};
        \node at (14 + 0.4, 0 + 0.75) {\scriptsize 1};
        \node at (15 + 0.6, 1 - 0.25) {\scriptsize 1};

        \draw[ultra thick, dotted, red] (5, 1) -- (6, 2) -- (7, 1);
        \node at (5 + 0.4, 1 + 0.75) {\color{red} \scriptsize 6};
        \node at (6 + 0.6, 2 - 0.25) {\color{red} \scriptsize 5};
    \end{tikzpicture}\\[0.2cm]
    $\downarrow$\\[0.4cm]
    \begin{tikzpicture}[scale = 0.75, yscale = 0.9]
        \draw[dotted] (0, 0) grid (16, 3);
        \draw[thick, blue!50] (0, 0) -- (1, 1) -- (2, 2) -- (3, 3) -- (4, 2) -- (5, 3) -- (6, 2) -- (7, 3) -- (8, 2) -- (9, 1) -- (10, 2) -- (11, 3) -- (12, 2) -- (13, 1) -- (14, 2) -- (15, 1) -- (16, 0);
        \draw[thick, orange] (0, 0) -- (1, 1) -- (2, 2) -- (3, 3) -- (4, 2) -- (5, 1) -- (6, 2) -- (7, 1) -- (8, 2) -- (9, 1) -- (10, 0) -- (11, 1) -- (12, 2) -- (13, 1) -- (14, 0) -- (15, 1) -- (16, 0);
        \node at (0 + 0.4, 0 + 0.75) {\scriptsize 8};
        \node at (1 + 0.4, 1 + 0.75) {\scriptsize 7};
        \node at (2 + 0.4, 2 + 0.75) {\scriptsize 5};
        \node at (3 + 0.6, 3 - 0.25) {\scriptsize 8};
        \node at (4 + 0.6, 2 - 0.25) {\scriptsize 7};
        \node at (5 + 0.4, 1 + 0.75) {\scriptsize 6};
        \node at (6 + 0.6, 2 - 0.25) {\scriptsize 5};
        \node at (7 + 0.4, 1 + 0.75) {\scriptsize 4};
        \node at (8 + 0.6, 2 - 0.25) {\scriptsize 6};
        \node at (9 + 0.6, 1 - 0.25) {\scriptsize 4};
        \node at (10 + 0.4, 0 + 0.75) {\scriptsize 3};
        \node at (11 + 0.4, 1 + 0.75) {\scriptsize 2};
        \node at (12 + 0.6, 2 - 0.25) {\scriptsize 3};
        \node at (13 + 0.6, 1 - 0.25) {\scriptsize 2};
        \node at (14 + 0.4, 0 + 0.75) {\scriptsize 1};
        \node at (15 + 0.6, 1 - 0.25) {\scriptsize 1};
    \end{tikzpicture}\\[0.2cm]
    $\downarrow$\\[0.4cm]
    \begin{tikzpicture}[scale = 0.75, yscale = 0.9]
        \draw[dotted] (0, 0) grid (16, 3);
        \draw[thick, blue!50] (0, 0) -- (1, 1) -- (2, 2) -- (3, 3) -- (4, 2) -- (5, 3) -- (6, 2) -- (7, 3) -- (8, 2) -- (9, 1) -- (10, 2) -- (11, 3) -- (12, 2) -- (13, 1) -- (14, 2) -- (15, 1) -- (16, 0);
        \draw[thick, orange] (0, 0) -- (1, 1) -- (2, 2) -- (3, 3) -- (4, 2) -- (5, 1) -- (6, 2) -- (7, 1) -- (8, 2) -- (9, 1) -- (10, 0) -- (11, 1) -- (12, 2) -- (13, 1) -- (14, 0) -- (15, 1) -- (16, 0);
        \node at (0 + 0.4, 0 + 0.75) {\scriptsize 8};
        \node[circle, draw = red, inner sep = 1pt] at (1 + 0.4, 1 + 0.75) {\scriptsize 7};
        \node at (2 + 0.4, 2 + 0.75) {\scriptsize 5};
        \node at (3 + 0.6, 3 - 0.25) {\scriptsize 8};
        \node at (4 + 0.6, 2 - 0.25) {\scriptsize 7};
        \node at (5 + 0.4, 1 + 0.75) {\scriptsize 6};
        \node at (6 + 0.6, 2 - 0.25) {\scriptsize 5};
        \node at (7 + 0.4, 1 + 0.75) {\scriptsize 4};
        \node[circle, draw = red, inner sep = 1pt] at (8 + 0.6, 2 - 0.25) {\scriptsize 6};
        \node at (9 + 0.6, 1 - 0.25) {\scriptsize 4};
        \node at (10 + 0.4, 0 + 0.75) {\scriptsize 3};
        \node at (11 + 0.4, 1 + 0.75) {\scriptsize 2};
        \node at (12 + 0.6, 2 - 0.25) {\scriptsize 3};
        \node at (13 + 0.6, 1 - 0.25) {\scriptsize 2};
        \node at (14 + 0.4, 0 + 0.75) {\scriptsize 1};
        \node at (15 + 0.6, 1 - 0.25) {\scriptsize 1};

        \draw[ultra thick, dotted, red] (4, 2) -- (5, 3) -- (6, 2);
        \node at (4 + 0.4, 2 + 0.75) {\color{red} \scriptsize 7};
        \node at (5 + 0.6, 3 - 0.25) {\color{red} \scriptsize 6};
    \end{tikzpicture}\\[0.2cm]
    $\downarrow$\\[0.2cm]
    $\vdots$\\[0.2cm]
    $\downarrow$\\[0.4cm]
    \begin{tikzpicture}[scale = 0.75, yscale = 0.9]
        \draw[dotted] (0, 0) grid (16, 3);
        \draw[thick, blue!50] (0, 0) -- (1, 1) -- (2, 2) -- (3, 3) -- (4, 2) -- (5, 3) -- (6, 2) -- (7, 3) -- (8, 2) -- (9, 1) -- (10, 2) -- (11, 3) -- (12, 2) -- (13, 1) -- (14, 2) -- (15, 1) -- (16, 0);
        \draw[thick, orange] (0, 0) -- (1, 1) -- (2, 2) -- (3, 3) -- (4, 2) -- (5, 3) -- (6, 2) -- (7, 3) -- (8, 2) -- (9, 1) -- (10, 2) -- (11, 3) -- (12, 2) -- (13, 1) -- (14, 0) -- (15, 1) -- (16, 0);
        \node at (0 + 0.4, 0 + 0.75) {\scriptsize 8};
        \node at (1 + 0.4, 1 + 0.75) {\scriptsize 6};
        \node at (2 + 0.4, 2 + 0.75) {\scriptsize 2};
        \node at (3 + 0.6, 3 - 0.25) {\scriptsize 8};
        \node at (4 + 0.4, 2 + 0.75) {\scriptsize 7};
        \node at (5 + 0.6, 3 - 0.25) {\scriptsize 6};
        \node at (6 + 0.4, 2 + 0.75) {\scriptsize 3};
        \node at (7 + 0.6, 3 - 0.25) {\scriptsize 2};
        \node at (8 + 0.6, 2 - 0.25) {\scriptsize 7};
        \node at (9 + 0.4, 1 + 0.75) {\scriptsize 5};
        \node at (10 + 0.4, 2 + 0.75) {\scriptsize 4};
        \node at (11 + 0.6, 3 - 0.25) {\scriptsize 3};
        \node at (12 + 0.6, 2 - 0.25) {\scriptsize 5};
        \node at (13 + 0.6, 1 - 0.25) {\scriptsize 4};
        \node at (14 + 0.4, 0 + 0.75) {\scriptsize 1};
        \node at (15 + 0.6, 1 - 0.25) {\scriptsize 1};
    \end{tikzpicture}\\[0.2cm]
    $\downarrow$\\[0.4cm]
    \begin{tikzpicture}[scale = 0.75, yscale = 0.9]
        \draw[dotted] (0, 0) grid (16, 3);
        \draw[thick, blue!50] (0, 0) -- (1, 1) -- (2, 2) -- (3, 3) -- (4, 2) -- (5, 3) -- (6, 2) -- (7, 3) -- (8, 2) -- (9, 1) -- (10, 2) -- (11, 3) -- (12, 2) -- (13, 1) -- (14, 2) -- (15, 1) -- (16, 0);
        \draw[thick, orange] (0, 0) -- (1, 1) -- (2, 2) -- (3, 3) -- (4, 2) -- (5, 3) -- (6, 2) -- (7, 3) -- (8, 2) -- (9, 1) -- (10, 2) -- (11, 3) -- (12, 2) -- (13, 1) -- (14, 0) -- (15, 1) -- (16, 0);
        \node at (0 + 0.4, 0 + 0.75) {\scriptsize 8};
        \node at (1 + 0.4, 1 + 0.75) {\scriptsize 6};
        \node[circle, draw = red, inner sep = 1pt] at (2 + 0.4, 2 + 0.75) {\scriptsize 2};
        \node at (3 + 0.6, 3 - 0.25) {\scriptsize 8};
        \node at (4 + 0.4, 2 + 0.75) {\scriptsize 7};
        \node at (5 + 0.6, 3 - 0.25) {\scriptsize 6};
        \node[circle, draw = red, inner sep = 1pt] at (6 + 0.4, 2 + 0.75) {\scriptsize 3};
        \node[circle, draw = red, inner sep = 1pt] at (7 + 0.6, 3 - 0.25) {\scriptsize 2};
        \node at (8 + 0.6, 2 - 0.25) {\scriptsize 7};
        \node at (9 + 0.4, 1 + 0.75) {\scriptsize 5};
        \node[circle, draw = red, inner sep = 1pt] at (10 + 0.4, 2 + 0.75) {\scriptsize 4};
        \node[circle, draw = red, inner sep = 1pt] at (11 + 0.6, 3 - 0.25) {\scriptsize 3};
        \node at (12 + 0.6, 2 - 0.25) {\scriptsize 5};
        \node at (13 + 0.6, 1 - 0.25) {\scriptsize 4};
        \node at (14 + 0.4, 0 + 0.75) {\scriptsize 1};
        \node[circle, draw = red, inner sep = 1pt] at (15 + 0.6, 1 - 0.25) {\scriptsize 1};
        
        \draw[ultra thick, dotted, red] (13, 1) -- (14, 2) -- (15, 1);
        \node at (13 + 0.4, 1 + 0.75) {\color{red} \scriptsize 4};
        \node at (14 + 0.6, 2 - 0.25) {\color{red} \scriptsize 3};
    \end{tikzpicture}\\[0.2cm]
    $\downarrow$\\[0.4cm]
    \begin{tikzpicture}[scale = 0.75, yscale = 0.9]
        \draw[dotted] (0, 0) grid (16, 3);
        \draw[thick, blue] (0, 0) -- (1, 1) -- (2, 2) -- (3, 3) -- (4, 2) -- (5, 3) -- (6, 2) -- (7, 3) -- (8, 2) -- (9, 1) -- (10, 2) -- (11, 3) -- (12, 2) -- (13, 1) -- (14, 2) -- (15, 1) -- (16, 0);
        \node at (0 + 0.4, 0 + 0.75) {\scriptsize 8};
        \node at (1 + 0.4, 1 + 0.75) {\scriptsize 6};
        \node at (2 + 0.4, 2 + 0.75) {\scriptsize 1};
        \node at (3 + 0.6, 3 - 0.25) {\scriptsize 8};
        \node at (4 + 0.4, 2 + 0.75) {\scriptsize 7};
        \node at (5 + 0.6, 3 - 0.25) {\scriptsize 6};
        \node at (6 + 0.4, 2 + 0.75) {\scriptsize 2};
        \node at (7 + 0.6, 3 - 0.25) {\scriptsize 1};
        \node at (8 + 0.6, 2 - 0.25) {\scriptsize 7};
        \node at (9 + 0.4, 1 + 0.75) {\scriptsize 5};
        \node at (10 + 0.4, 2 + 0.75) {\scriptsize 3};
        \node at (11 + 0.6, 3 - 0.25) {\scriptsize 2};
        \node at (12 + 0.6, 2 - 0.25) {\scriptsize 5};
        \node at (13 + 0.4, 1 + 0.75) {\scriptsize 4};
        \node at (14 + 0.6, 2 - 0.25) {\scriptsize 3};
        \node at (15 + 0.6, 1 - 0.25) {\scriptsize 4};
    \end{tikzpicture}
    \caption{\Cref{valgo} applied to a Dyck path.}
    \label{fig:valgoex}
\end{figure}

SageMath \cite{sagemath} code implementing \Cref{valgo} to return $\vperm d$ can be found at
\begin{center}
    \href{https://github.com/KrishnaMenon1998/Bouncing-canons}{\texttt{https://github.com/KrishnaMenon1998/Bouncing-canons}}.
\end{center}

\begin{theorem}\label{thm:vperm}
    For any Dyck path $d$, we have $\vperm d \in M_d$.
\end{theorem}

\begin{proof}
    We will prove the following three statements:
    \begin{enumerate}
        \item Suppose that $c'$ is obtained by performing a valley swap to a canon permutation $c$ where the down-step has label $j$ and the up-step has label $i$ where $i < j$. 
        The position immediately before the second copy of $i$ in $c$ is the only place where a descent in $c$ might not remain a descent in $c'$.\label{s1}

        \item Any valley being swapped has a larger label on its down-step than its up-step, when performing \Cref{valgo}.\label{s2}

        \item During any stage in \Cref{valgo}, each bounce factor is labeled with a decreasing sequence.\label{s3}
    \end{enumerate}
    Statement \ref{s1} is an easy consequence of the definition of valley swaps. 
    Just as for $\bperm d$, statement \ref{s3} will prove the result. 
    However, we will also need statement \ref{s2} in order to apply statement \ref{s1} in what follows. 
    Also, note that any Dyck path appearing in the course of \Cref{valgo} has the same bounce path.

    During \Cref{valgo}, all required valleys in the second bounce factor are swapped, and then those in the third, and so on (the first bounce factor has no valleys). 
    Also note that any up-step in a bounce factor has its corresponding down-step in the next bounce factor. 
    Hence, if statement \ref{s2} is true, then statement \ref{s1} implies that performing valley swaps in some bounce factor during the algorithm will not change the fact that the previous bounce factors are labeled with decreasing sequences.

    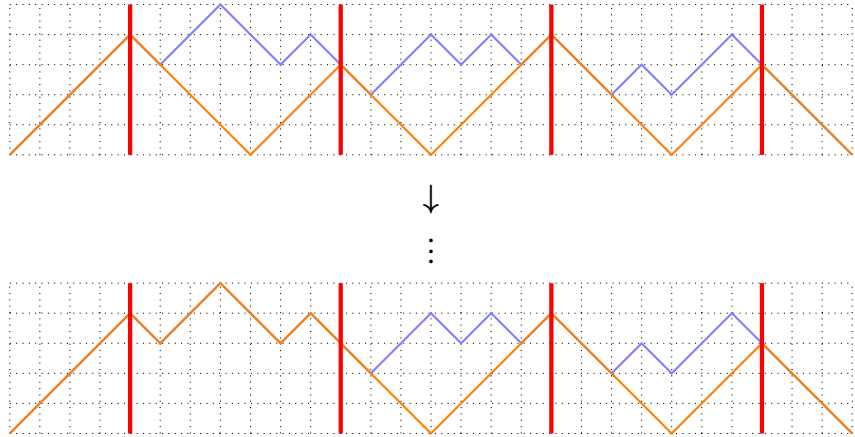
\begin{figure}[H]
        \centering
        \begin{tikzpicture}[scale = 0.4]
            \draw[dotted] (0, 0) grid (28, 5);
            \draw[thick,blue!50] (0, 0) -- (1, 1) -- (2, 2) -- (3, 3) -- (4, 4) -- (5, 3) -- (6, 4) -- (7, 5) -- (8, 4) -- (9, 3) -- (10, 4) -- (11, 3) -- (12, 2) -- (13, 3) -- (14, 4) -- (15, 3) -- (16, 4) -- (17, 3) -- (18, 4) -- (19, 3) -- (20, 2) -- (21, 3) -- (22, 2) -- (23, 3) -- (24, 4) -- (25, 3) -- (26, 2) -- (27, 1) -- (28, 0);
            \draw[thick,orange] (0, 0) -- (1, 1) -- (2, 2) -- (3, 3) -- (4, 4) -- (5, 3) -- (6, 2) -- (7, 1) -- (8, 0) -- (9, 1) -- (10, 2) -- (11, 3) -- (12, 2) -- (13, 1) -- (14, 0) -- (15, 1) -- (16, 2) -- (17, 3) -- (18, 4) -- (19, 3) -- (20, 2) -- (21, 1) -- (22, 0) -- (23, 1) -- (24, 2) -- (25, 3) -- (26, 2) -- (27, 1) -- (28, 0);

            \draw[ultra thick,red] (4, 0) -- (4, 5);
            \draw[ultra thick,red] (11, 0) -- (11, 5);
            \draw[ultra thick,red] (18, 0) -- (18, 5);
            \draw[ultra thick,red] (25, 0) -- (25, 5);
        \end{tikzpicture}\\[0.1cm]
        $\downarrow$\\[0.1cm]
        $\vdots$\\[0.2cm]
        \begin{tikzpicture}[scale = 0.4]
            \draw[dotted] (0, 0) grid (28, 5);
            \draw[thick,blue!50] (0, 0) -- (1, 1) -- (2, 2) -- (3, 3) -- (4, 4) -- (5, 3) -- (6, 4) -- (7, 5) -- (8, 4) -- (9, 3) -- (10, 4) -- (11, 3) -- (12, 2) -- (13, 3) -- (14, 4) -- (15, 3) -- (16, 4) -- (17, 3) -- (18, 4) -- (19, 3) -- (20, 2) -- (21, 3) -- (22, 2) -- (23, 3) -- (24, 4) -- (25, 3) -- (26, 2) -- (27, 1) -- (28, 0);
            \draw[thick,orange] (0, 0) -- (1, 1) -- (2, 2) -- (3, 3) -- (4, 4) -- (5, 3) -- (6, 4) -- (7, 5) -- (8, 4) -- (9, 3) -- (10, 4) -- (11, 3) -- (12, 2) -- (13, 1) -- (14, 0) -- (15, 1) -- (16, 2) -- (17, 3) -- (18, 4) -- (19, 3) -- (20, 2) -- (21, 1) -- (22, 0) -- (23, 1) -- (24, 2) -- (25, 3) -- (26, 2) -- (27, 1) -- (28, 0);
            \draw[ultra thick,red] (4, 0) -- (4, 5);
            \draw[ultra thick,red] (11, 0) -- (11, 5);
            \draw[ultra thick,red] (18, 0) -- (18, 5);
            \draw[ultra thick,red] (25, 0) -- (25, 5);
        \end{tikzpicture}
        \caption{Suppose all valleys in the second bounce factor have been swapped while performing \Cref{valgo} (labels omitted). We have to show that the labels in the third bounce factor still form a decreasing sequence.}
        \label{fig:videa1}
    \end{figure}

    We will show that during \Cref{valgo}, after swapping all the valleys in one bounce factor, the labels in the next bounce factor remain a decreasing sequence. 
    Combining this with the observations above will prove the result (see \Cref{fig:videa1}). 
    Since we only care about the relative order of the labels, we can reduce the problem to proving statements \ref{s2} and \ref{s3} for when $\bpk d = 2$.

    To do this, we note how valley swapping changes the labels on the valley being swapped. 
    Suppose the down-step has label $j$ and the up-step has label $i$ with $i < j$. 
    After the valley swap, the down-step becomes an up-step labeled $j$ and the up-step becomes a down-step labeled $j - 1$. 
    Filling in the grid between $d$ and the bounce path, we can use this idea to see what labels the various steps along the way have during the algorithm (see \Cref{fig:videa2}). 
    We use $b_1, b_2, \ldots$ to label the down-steps in the second bounce factor and use $a_1, a_2, \ldots$ for the up-steps. 
    If $d \in \Dyck_n$, we actually have $b_1 = n, b_2 = n - 1, \ldots$ but we avoid these labels to show that we only use the fact that $b_1 > b_2 > \cdots > a_1 > a_2 > \cdots$ (relative order of the labels).

    \begin{figure}[H]
        \centering
        \begin{tikzpicture}[scale = 1.2]
            \draw[dotted] (4, 0) grid (15, 5);
            \draw[thick,blue!50] (7, 3) -- node[above,sloped] {\color{black}\tiny $b_3$} (8, 4) -- node[above,sloped] {\color{black}\tiny $b_3 - 1$}(9, 5) -- node[above,sloped] {\color{black}\tiny $b_3 - 2$}(10, 4) -- node[above,sloped] {\color{black}\tiny $b_4 - 2$}(11, 5) -- node[above,sloped] {\color{black}\tiny $b_4 - 3$}(12, 4) -- node[above,sloped] {\color{black}\tiny $b_5 - 3$}(13, 3);
            \draw[thick, orange] (4, 4) -- (5, 5) --node[above,sloped] {\color{black}\tiny $b_1$} (6, 4) --node[above,sloped] {\color{black}\tiny $b_2$} (7, 3) --node[above,sloped] {\tiny $b_3$} (8, 2) --node[above,sloped] {\tiny $b_4$} (9, 1) --node[above,sloped] {\tiny $b_5$} (10, 0) --node[above,sloped] {\tiny $a_1$} (11, 1) -- node[above,sloped] {\tiny $a_2$}(12, 2) -- node[above,sloped] {\tiny $a_3$}(13, 3) -- node[above,sloped] {\color{black} \tiny $a_4$}(14, 4) -- (15, 3);
            \draw[thick,dashed,orange] (8, 2) -- node[above,sloped] {\tiny $b_4$}(9, 3) -- node[above,sloped] {\tiny $b_4 - 1$}(10, 4);
            \draw[thick,dashed,orange] (9, 1) -- node[above,sloped] {\tiny $b_5$}(10, 2) -- node[above,sloped] {\tiny $b_5 - 1$}(11, 3) -- node[above,sloped] {\tiny $b_5 - 2$} (12, 4);
            \draw[thick,dashed,orange] (11, 1) -- node[above,sloped] {\tiny $b_5 - 1$}(10, 2) -- node[above,sloped] {\tiny $b_4 - 1$}(9, 3) -- node[above,sloped] {\tiny $b_3 - 1$}(8, 4);
            \draw[thick,dashed,orange] (10, 4) -- node[above,sloped] {\tiny $b_4 - 2$}(11, 3) -- node[above,sloped] {\tiny $b_5 - 2$}(12, 2);
            \draw[ultra thick, red] (5, 0) -- (5, 5);
            \draw[ultra thick, red] (14, 0) -- (14, 5);
        \end{tikzpicture}
        \caption{How the labels in a bounce factor change as the valleys in it are swapped. The final labels are in black}
        \label{fig:videa2}
    \end{figure}

    From \Cref{fig:videa2}, it can be seen that any valley that is swapped during the algorithm must have a larger label on its down-step. 
    Also, the labels in the second bounce factor end up in a decreasing sequence. 
    In particular, the labels on the up-steps in the second bounce factor, which match the sequence of labels of the down-steps in the third bounce factor, are decreasing. 
    This proves statements \ref{s2} and \ref{s3} when $\bpk d = 2$ and hence the theorem.
\end{proof}

An example of a Dyck path $d$ for which $\bperm d \neq \vperm d$ is $d = U(UD)^4D$ where $\bperm d = 52413$ and $\vperm d = 53412$. Even though in this simple example, $\vperm$ can arise from $\bperm$ with a single transposition, the difference between $\bperm$ and $\vperm$ becomes more apparent for larger Dyck paths with more than $2$ components. For instance, for $d = (UUD)^2DUD^2U(UD)^2U^2D^3U(UD)^2D$, we compute $\bperm d = 13\ 1\ 12\ 2\ 11\ 10\ 8\ 9\ 7\ 6\ 5\ 3\ 4$ and $\vperm d = 13\ 9\ 12\ 10\ 11\ 8\ 6\ 7\ 5\ 4\ 3\ 1\ 2$.

In the next section, we will prove that $\bperm d $ and $\vperm d$ are in fact two different linear extensions of an $n$-element poset, which depends on $d$ and $\bounce d$. This fact, together with an extension of the method for computing $\vperm d$ 
 will allow us to characterize the permutations that contribute to the leading coefficient of the descent polynomial $C_d$.

\begin{remark}\label{genbperm}
One could define a generalization of $\bperm$ that covers more than one maximizer to include both $\bperm$ and $\vperm$. 
Instead of labeling the first unlabeled step with $i-1$ in Algorithm \ref{algo}, one could use a label $i-k$ as long as both copies of $i-1,\ldots,i-k+1$ appear to the left of the up-step pending labeling. 
We do not provide a proof of this but we have verified it with code up to semilength $n = 14$.
\end{remark}

\section{The leading coefficient of $C_d$}\label{sec: interpret leader}

So far we have discussed two methods that provide two, often different, canon permutations with a maximum number of descents associated to a given Dyck path $d$. 
It would be natural to wonder if similar methods for different paths under $d$ could provide more maximizers, i.e., permutations that contribute to the leading coefficient of $C_d$. 
One immediate example is the \emph{reverse bounce path} of $d$, which arises by reversing $d$, computing its bounce path and reversing it back. 
Our methods for computing $\bperm d$ and $\vperm d$ could be adjusted to the {reverse bounce path} of $d$ and yield more maximizers. 
Below, we identify a set of Dyck paths under $d$, such that all maximizers of $d$ can arise through methods similar to $\vperm$ and $\bperm$ applied to these paths.

\begin{definition}
    For $d \in \Dyck_n$, define $B_d$ to be the set of $b \in \Dyck_n$ such that
    \begin{itemize}
        \item $b$ is under $d$,
        \item $\pk b = \bpk d$ with each peak of $b$ touching $d$, and
        \item each valley of $b$ is either also a valley of $d$ or is at height $0$.
    \end{itemize}
\end{definition}
Another way to view the Dyck paths in $B_d$ is as follows: 
Dyck paths $b$ under $d$ with $\pk b = \bpk d$ such that if $b$ deviates from $d$, then $b$ must fall to the $x$-axis, and then rise until it touches $d$ again. 
An example of a path in $B_d$ is shown in \Cref{fig:Bdex1}.

\begin{figure}[H]
    \centering
    \begin{tikzpicture}[scale = 0.6]
        \draw[dotted] (0, 0) grid (16, 4);
        \draw[thick, blue!50] (0, 0) -- (1, 1) -- (2, 2) -- (3, 1) -- (4, 2) -- (5, 3) -- (6, 4) -- (7, 3) -- (8, 2) -- (9, 3) -- (10, 2) -- (11, 3) -- (12, 2) -- (13, 1) -- (14, 2) -- (15, 1) -- (16, 0);
        \draw[red] (0, 0) -- (1, 1) -- (2, 2) -- (3, 1) -- (4, 2) -- (5, 3) -- (6, 2) -- (7, 1) -- (8, 0) -- (9, 1) -- (10, 2) -- (11, 1) -- (12, 0) -- (13, 1) -- (14, 2) -- (15, 1) -- (16, 0);
    \end{tikzpicture}
    \caption{A Dyck path $\textcolor{blue}{d}$ and a path $\textcolor{red}{b}\in B_d$ under it.}
    \label{fig:Bdex1}
\end{figure}

Both $\bounce d$ as well as the reverse bounce path of $d$ (which can be the same as $\bounce d$) lie in $B_d$. 
Note that $d$ is in $B_d$ precisely when $\pk d = \bpk d$, i.e., when the decreasing permutation is in $M_d$ (\Cref{cor: descreasing is maximizer}).

\begin{example}\label{ex:Bd}
    For $d = U^2DU^2D^2(UD)^2D$, $B_d$ has $3$ paths, shown in \Cref{fig:Bdex2}.
\end{example}

\begin{figure}[H]
    \centering
    \begin{tikzpicture}[scale = 0.5]
        \draw[dotted] (0, 0) grid (12, 3);
        \draw[thick, blue!50] (0, 0) -- (1, 1) -- (2, 2) -- (3, 1) -- (4, 2) -- (5, 3) -- (6, 2) -- (7, 1) -- (8, 2) -- (9, 1) -- (10, 2) -- (11, 1) -- (12, 0);
        \draw[red] (0, 0) -- (1, 1) -- (2, 2) -- (3, 1) -- (4, 0) -- (5, 1) -- (6, 2) -- (7, 1) -- (8, 0) -- (9, 1) -- (10, 2) -- (11, 1) -- (12, 0);
    \end{tikzpicture}
    \vspace{0.3cm}
    
    \begin{tikzpicture}[scale = 0.5]
        \draw[dotted] (0, 0) grid (12, 3);
        \draw[thick, blue!50] (0, 0) -- (1, 1) -- (2, 2) -- (3, 1) -- (4, 2) -- (5, 3) -- (6, 2) -- (7, 1) -- (8, 2) -- (9, 1) -- (10, 2) -- (11, 1) -- (12, 0);
        \draw[red] (0, 0) -- (1, 1) -- (2, 2) -- (3, 1) -- (4, 2) -- (5, 3) -- (6, 2) -- (7, 1) -- (8, 0) -- (9, 1) -- (10, 2) -- (11, 1) -- (12, 0);
    \end{tikzpicture}
    \hspace{0.2cm}
    \begin{tikzpicture}[scale = 0.5]
        \draw[dotted] (0, 0) grid (12, 3);
        \draw[thick, blue!50] (0, 0) -- (1, 1) -- (2, 2) -- (3, 1) -- (4, 2) -- (5, 3) -- (6, 2) -- (7, 1) -- (8, 2) -- (9, 1) -- (10, 2) -- (11, 1) -- (12, 0);
        \draw[red] (0, 0) -- (1, 1) -- (2, 0) -- (3, 1) -- (4, 2) -- (5, 3) -- (6, 2) -- (7, 1) -- (8, 0) -- (9, 1) -- (10, 2) -- (11, 1) -- (12, 0);
    \end{tikzpicture}
    \caption{The paths in $B_d$ for $d$ as in \Cref{ex:Bd}.}
    \label{fig:Bdex2}
\end{figure}
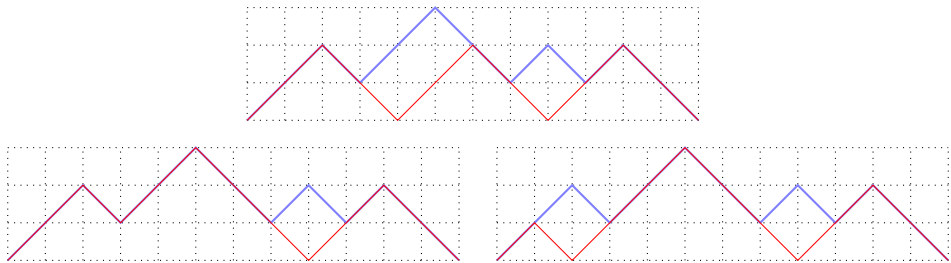

\begin{theorem}\label{Dessetchar}
    For any Dyck path $d \in \Dyck_n$, we have
    \begin{equation*}
        \{\Des(d, \sigma) \mid \sigma \in M_d\} = \{\Des(b,\delta_n) \mid b \in B_d\}.
    \end{equation*}
    That is, the possible descent sets for descent-maximizers of $d$ are the same as the possible peak sets for Dyck paths in $B_d$.
\end{theorem}

\begin{proof}
    The valley-swapping algorithm (\cref{valgo}) can be extended from $\bounce d$ to any Dyck path $b$ in $B_d$, where valleys are swapped (from left to right) until $b$ transforms into $d$. 
    It can be seen in a similar way as in \cref{sec:valleys} that the output is again a permutation $\sigma_b$ that lies in $M_d$, and the descent sets of $\can (b, \delta_n)$ and $\can (d, \sigma_b)$ coincide. 
    This shows that $\{\Des(b,\delta_n) \mid b \in B_d\} \subseteq \{\Des(d, \sigma) \mid \sigma \in M_d\}$.
    
    To complete the proof, we show that $\Des(d, \sigma)$ matches the peaks of a path in $B_d$ for any $\sigma \in M_d$. 
    To do this, we start with $\sigma \in M_d$, label $d$ with $\sigma$, mark the non-descents using {(red)} dots, and (try to) construct a {(red)} Dyck path $b$ whose peaks exactly match the dots (see \Cref{fig:constBd}). 
    Note that the labels between any two consecutive dots form a decreasing sequence.

    \begin{figure}[h]
        \centering
        \begin{tikzpicture}[scale = 0.75]
            \draw[dotted] (0, 0) grid (10, 3);
            \draw[thick, blue] (0, 0) -- (1, 1) -- (2, 2) -- (3, 1) -- (4, 2) -- (5, 3) -- (6, 2) -- (7, 1) -- (8, 2) -- (9, 1) -- (10, 0);
            \node at (0 + 0.4, 0 + 0.75) {\scriptsize 4};
            \node at (1 + 0.4, 1 + 0.75) {\scriptsize 3};
            \node at (2 + 0.6, 2 - 0.25) {\scriptsize 4};
            \node at (3 + 0.4, 1 + 0.75) {\scriptsize 2};
            \node at (4 + 0.4, 2 + 0.75) {\scriptsize 5};
            \node at (5 + 0.6, 3 - 0.25) {\scriptsize 3};
            \node at (6 + 0.6, 2 - 0.25) {\scriptsize 2};
            \node at (7 + 0.4, 1 + 0.75) {\scriptsize 1};
            \node at (8 + 0.6, 2 - 0.25) {\scriptsize 5};
            \node at (9 + 0.6, 1 - 0.25) {\scriptsize 1};

            \node [circle, fill = red, inner sep = 2pt] at (2, 2) {};
            \node [circle, fill = red, inner sep = 2pt] at (4, 2) {};
            \node [circle, fill = red, inner sep = 2pt] at (8, 2) {};
        \end{tikzpicture}\\
        $\downarrow$\\[0.2cm]
        \begin{tikzpicture}[scale = 0.75]
            \draw[dotted] (0, 0) grid (10, 3);
            \draw[thick, blue!50] (0, 0) -- (1, 1) -- (2, 2) -- (3, 1) -- (4, 2) -- (5, 3) -- (6, 2) -- (7, 1) -- (8, 2) -- (9, 1) -- (10, 0);
            \node at (0 + 0.4, 0 + 0.75) {\scriptsize 4};
            \node at (1 + 0.4, 1 + 0.75) {\scriptsize 3};
            \node at (2 + 0.6, 2 - 0.25) {\scriptsize 4};
            \node at (3 + 0.4, 1 + 0.75) {\scriptsize 2};
            \node at (4 + 0.4, 2 + 0.75) {\scriptsize 5};
            \node at (5 + 0.6, 3 - 0.25) {\scriptsize 3};
            \node at (6 + 0.6, 2 - 0.25) {\scriptsize 2};
            \node at (7 + 0.4, 1 + 0.75) {\scriptsize 1};
            \node at (8 + 0.6, 2 - 0.25) {\scriptsize 5};
            \node at (9 + 0.6, 1 - 0.25) {\scriptsize 1};

            \node [circle, fill = red, inner sep = 2pt] at (2, 2) {};
            \node [circle, fill = red, inner sep = 2pt] at (4, 2) {};
            \node [circle, fill = red, inner sep = 2pt] at (8, 2) {};

            \draw[red, thin] (0, 0) -- (1, 1) -- (2, 2) -- (3, 1) -- (4, 2) -- (5, 1) -- (6, 0) -- (7, 1) -- (8, 2) -- (9, 1) -- (10, 0);
        \end{tikzpicture}
        \caption{Constructing \textcolor{red}{$b$} from 
    \textcolor{blue}{$d$} using the non-descents of $\can(d,\sigma)$.}
        \label{fig:constBd}
    \end{figure}
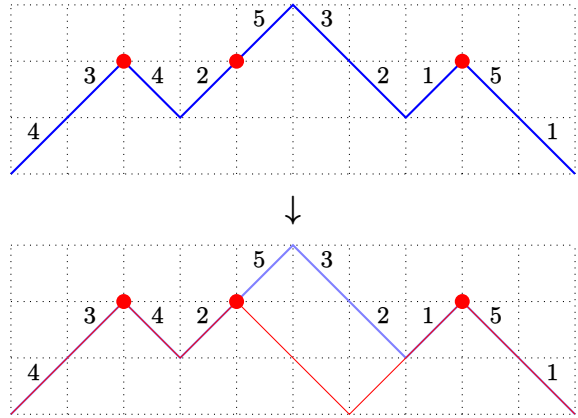

    We need to check that it is possible to construct $b$, and then confirm that it lies in $B_d$. 
    Note that since $\sigma \in M_d$, if $b$ exists, we must have $\pk b = \bpk d$ and $b$ must have exactly one peak in each bounce factor of $d$ except the last bounce factor. 
    More precisely, in terms of the dots corresponding to non-descents, there must be a dot at the end of precisely one of the steps in each bounce factor of $d$, apart from the last factor. 
    We first show that $b$ exists. 
    We only have to show that it is possible to reach from one dot to the next using a sequence of down-steps followed by a sequence of up-steps.

    Let $P$ be a dot corresponding to a non-descent and $Q$ be the dot following it. 
    Let $l_P$ be the sequence of down-steps to the $x$-axis starting from $P$. 
    Also, set $l_Q$ to be the sequence of up-steps that starts at the $x$-axis and ends at $Q$. 
    By the conditions mentioned above, $Q$ cannot lie on $l_P$ and $P$ cannot lie on $l_Q$. 
    Hence, the Dyck path $b$ exists if $l_P$ and $l_Q$ meet at some point. 
    This not being the case corresponds to $l_P$ touching the $x$-axis at a point before $l_Q$ (see \Cref{fig:bexists}). 
    But this would mean that there is an up-step and its corresponding down-step between $P$ and $Q$ (see $\star$ in \Cref{fig:bexists}). 
    This would imply that $\can(d, \sigma)$ will have another non-descent between $P$ and $Q$, which is a contradiction. 
    Hence, there exists a Dyck path $b$ whose peaks match the dots on $d$.

    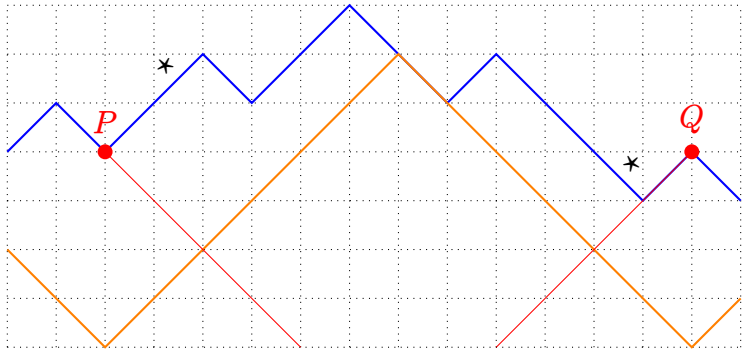
\begin{figure}[H]
        \centering
        \begin{tikzpicture}[scale = 0.65]
            \draw[dotted] (2, -2) grid (17, 5);
            \draw[thick, blue] (2, 2) -- (3, 3) -- (4, 2) -- (5, 3) -- node[above, sloped] {\color{black} $\star$} (6, 4) -- (7, 3) -- (8, 4) -- (9, 5) -- (10, 4) -- (11, 3) -- (12, 4) -- (13, 3) -- (14, 2) -- node[above, sloped] {\color{black} $\star$} (15, 1) -- (16, 2) -- (17, 1);
            \begin{scope}[shift = {(-2,-2)}]
                \draw[thick, orange] (4, 2) -- (5, 1) -- (6, 0) -- (7, 1) -- (8, 2) -- (9, 3) -- (10, 4) -- (11, 5) -- (12, 6) -- (13, 5) -- (14, 4) -- (15, 3) -- (16, 2) -- (17, 1) -- (18, 0) -- (19, 1);
            \end{scope}
            \node[circle, inner sep = 2pt, fill = red, label=above:$\textcolor{red}{P}$] at (4, 2) {};
            \node[circle, inner sep = 2pt, fill = red, label=above:$\textcolor{red}{Q}$] at (16, 2) {};
            \draw[thin, red] (4, 2) -- (8, -2);
            \draw[thin, red] (12, -2) -- (16, 2);
        \end{tikzpicture}
        \caption{If $l_P$ and $l_Q$ do not meet, there is a pair of corresponding steps between $P$ and $Q$ (one such pair is marked using $\star$).}
        \label{fig:bexists}
    \end{figure}

    We now show that $b$ is in $B_d$. 
    Just as before, set $P$ and $Q$ to be consecutive peaks of $b$. 
    We define $l_P$ (resp. $l_Q$) to be the down-steps (resp. up-steps) of $b$ between $P$ and $Q$. 
    We have to show that either $l_P$ and $l_Q$ intersect at the $x$-axis, or that the steps of $l_P$ and $l_Q$ are also steps of $d$. 
    Let $F_P$ be the bounce factor of $d$ containing the step whose endpoint is $P$ and similarly define $F_Q$. 
    We consider two cases based on whether or not there is an up-step of $F_P$ between $P$ and $Q$. 
    We first assume that there is. 
    By way of contradiction, assume that $l_P$ and $l_Q$ meet at a point above the $x$-axis (see \Cref{fig:Pupstepbetween}).

    \begin{figure}[H]
        \centering
        \begin{tikzpicture}[scale = 0.65]
            \draw[dotted] (2, -2) grid (17, 5);
            \draw[thick, blue] (2, 2) -- (3, 3) -- (4, 2) -- (5, 3) -- (6, 4) -- (7, 3) -- (8, 4) -- (9, 5) -- (10, 4) -- (11, 3) -- (12, 4) -- (13, 3) -- (14, 4) -- (15, 3) -- (16, 2) -- (17, 1);
            \begin{scope}[shift = {(-2,-2)}]
                \draw[thick, orange] (4, 2) -- (5, 1) -- (6, 0) -- node[above,sloped] {\color{black}\tiny $a_1$}(7, 1) -- node[above,sloped] {\color{black}\tiny $a_2$}(8, 2) -- node[above,sloped] {\color{black}\tiny $a_3$}(9, 3) -- node[above,sloped] {\color{black}\tiny $a_4$}(10, 4) -- node[above,sloped] {\color{black}\tiny $a_5$}(11, 5) -- node[above,sloped] {\color{black}\tiny $a_6$}(12, 6) -- node[above,sloped] {\color{black}\tiny $a_1$}(13, 5) -- node[above,sloped] {\color{black}\tiny $a_2$}(14, 4) -- node[above,sloped] {\color{black}\tiny $a_3$}(15, 3) -- node[above,sloped] {\color{black}\tiny $a_4$}(16, 2) -- node[above,sloped] {\color{black}\tiny $a_5$}(17, 1) -- node[above,sloped] {\color{black}\tiny $a_6$}(18, 0) -- (19, 1);
            \end{scope}
            \node[circle, inner sep = 2pt, fill = red, label=above:$\textcolor{red}{P}$] at (6, 4) {};
            \node[circle, inner sep = 2pt, fill = red, label=above:$\textcolor{red}{Q}$] at (15, 3) {};
            \draw[thin, red] (6, 4) -- (11, -1) -- (15, 3);
        \end{tikzpicture}
        \caption{Case when there is an up-step of $F_P$ after $P$ and the \textcolor{red}{valley} of $b$ between $P$ and $Q$ is not on the $x$-axis.}
        \label{fig:Pupstepbetween}
    \end{figure}
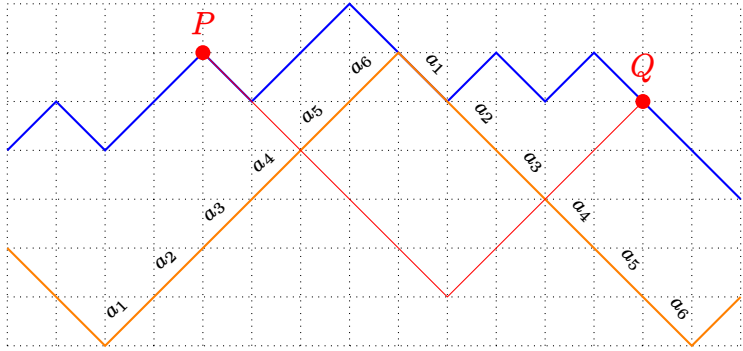

    Instead of writing the labels of $\sigma$ on $d$, we write them on $\bounce d$ as in \Cref{fig:Pupstepbetween}, in order to make arguments in the sequel clearer. 
    Let the labels on the primitive factor of $\bounce d$ whose peak is between $P$ and $Q$ be $a_1, a_2, \ldots, a_k$. 
    Suppose the line $l_P$ breaks the sequence of labels on the up-steps of $\bounce d$ between $a_j$ and $a_{j + 1}$ (such a $j$ exists since there is an up-step of $F_P$ after $P$). 
    Since $l_P$ and $l_Q$ meet above the $x$-axis, the line $l_Q$ must break the sequence of labels (on the down-steps) between $a_i$ and $a_{i + 1}$ for some $i < j$. 
    For the example in \Cref{fig:Pupstepbetween}, we have $i = 3$ and $j = 4$. 
    Since the labels $a_1, a_2, \ldots, a_j$ appear in order before $P$ in $F_P$, we get $a_1 > a_2 > \cdots > a_j$. 
    Similarly, looking at the down-steps after $Q$ in $F_Q$, we get $a_{i + 1} > a_{i + 2} > \cdots > a_k$. 
    This gives us $a_1 > a_2 > \cdots > a_k$. 
    But looking at the part of the canon permutation between $P$ and $Q$ gives us $a_k > a_1$, which is a contradiction.

    \begin{figure}[H]
        \centering
        \begin{tikzpicture}[scale = 0.6]
            \draw[dotted] (3, -3) grid (20, 5);
            \draw[thick, blue] (3, 3) -- (4, 4) -- (5, 3) -- (6, 4) -- (7, 5) -- (8, 4) -- (9, 3) -- (10, 2) -- (11, 1) -- (12, 2) -- (13, 3) -- (14, 2) -- (15, 3) -- (16, 2) -- (17, 1) -- (18, 2) -- (19, 1) -- (20, 0);
            \begin{scope}[shift = {(3, -3)}]
                \draw[thick, orange] (0, 0) -- (1, 1) -- (2, 2) -- (3, 3) -- node[above,sloped] {\color{black}\tiny $a_1$} (4, 4) -- node[above,sloped] {\color{black}\tiny $a_2$} (5, 5) -- (6, 6) -- (7, 5) -- (8, 4) -- (9, 3) -- node[above,sloped] {\color{black}\tiny $a_1$} (10, 2) -- node[above,sloped] {\color{black}\tiny $a_2$} (11, 1) -- (12, 0) -- node[above,sloped] {\color{black}\tiny $a_3$} (13, 1) -- (14, 2) -- (15, 3) -- (16, 4) -- node[above,sloped] {\color{black}\tiny $a_3$} (17, 3);
            \end{scope}
            \node [circle, fill = red, inner sep = 2pt, label = $\textcolor{red}{P}$] at (8, 4) {};
            \node [circle, fill = red, inner sep = 2pt, label = $\textcolor{red}{Q}$] at (16, 2) {};
            \draw[thin, red] (8, 4) -- (13, -1) -- (16, 2);
        \end{tikzpicture}
        \caption{Case when there are no up-steps of $F_P$ after $P$ and the \textcolor{red}{valley} of $b$ between $P$ and $Q$ is neither on the $x$-axis nor a valley of~$d$.}
        \label{fig:Pnoupstepbetween}
    \end{figure}
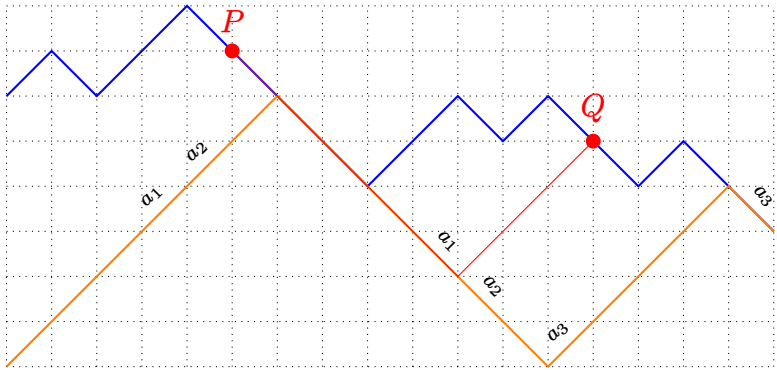

    Next, suppose there are no up-steps of $F_P$ between $P$ and $Q$. 
    By way of contradiction, assume that $l_P$ and $l_Q$ meet above the $x$-axis and there is a peak of $d$ between $P$ and $Q$. 
    Let $a_1, a_2$ be the labels where $l_Q$ splits the down-steps of the primitive factor of $\bounce d$ between $P$ and $Q$ (see \Cref{fig:Pnoupstepbetween}). 
    Let $a_3$ be the label of the first up-step in $F_Q$. 
    Looking at the steps before $P$ in $F_P$, we get $a_1 > a_2$. 
    Looking at the steps between $P$ and $Q$, we get $a_3 > a_1$. 
    However, the part of the canon permutation after $Q$ and before the next peak of $b$ has a copy of $a_2$ before one of $a_3$. 
    Since $a_2 < a_3$, there is a non-descent in $F_Q$ apart from $Q$, which is a contradiction.
\end{proof}

The theorem above characterizes the possible descent sets of $\can(d, \sigma)$ for $\sigma \in M_d$. 
We can use this to view the permutations in $M_d$ as linear extensions of certain posets. 
For any $d \in \Dyck_n$ and $b \in B_d$, we construct the poset $P = P_{d, b}$ on the symbols $\{a_1, \ldots, a_n\}$. 
We start by labeling $d$ using the identity permutation to get $\can(d, 12 \cdots n)$ and replacing the label $i$ with $a_i$. 
We now set $a_i >_P a_j$ if the labels $a_i$ and $a_j$ appear between two consecutive peaks of $b$ with $a_i$ appearing before (i.e., to the left of) $a_j$. 
The partial order on $P$ is the transitive closure of this relation. 
Note that $P$ is a valid poset by \Cref{Dessetchar}.

\begin{example}
    For $d = U(UD)^4D$ and $b = U^2D^2UDU^2D^2 \in B_d$, the order relations of $P = P_{d, b}$ imposed by peaks of $b$ are
    \begin{equation*}
        a_1 >_P a_2, \ \ a_1 >_P a_3 >_P a_2, \ \ a_4 >_P a_3 >_P a_5, \text{ and } a_4 \mathrel{{}_P{>}} a_5.
    \end{equation*}
    The Hasse diagram of $P$ is given in \Cref{fig:posetex}.
\end{example}

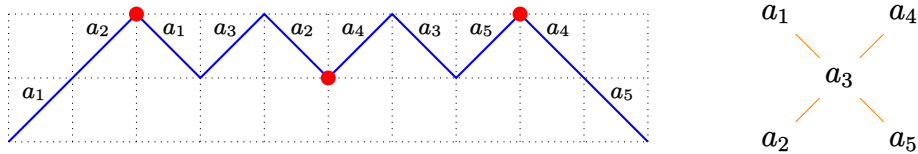
\begin{figure}[H]
    \centering
    \begin{tikzpicture}[scale = 0.85]
        \draw[dotted] (0, 0) grid (10, 2);
        \draw[thick, blue] (0, 0) -- (1, 1) -- (2, 2) -- (3, 1) -- (4, 2) -- (5, 1) -- (6, 2) -- (7, 1) -- (8, 2) -- (9, 1) -- (10, 0);
        \node[circle, fill = red, inner sep = 2pt] at (2, 2) {};
        \node[circle, fill = red, inner sep = 2pt] at (5, 1) {};
        \node[circle, fill = red, inner sep = 2pt] at (8, 2) {};

        \node at (0 + 0.4, 0 + 0.75) {\scriptsize $a_1$};
        \node at (1 + 0.4, 1 + 0.75) {\scriptsize $a_2$};
        \node at (2 + 0.6, 2 - 0.25) {\scriptsize $a_1$};
        \node at (3 + 0.4, 1 + 0.75) {\scriptsize $a_3$};
        \node at (4 + 0.6, 2 - 0.25) {\scriptsize $a_2$};
        \node at (5 + 0.4, 1 + 0.75) {\scriptsize $a_4$};
        \node at (6 + 0.6, 2 - 0.25) {\scriptsize $a_3$};
        \node at (7 + 0.4, 1 + 0.75) {\scriptsize $a_5$};
        \node at (8 + 0.6, 2 - 0.25) {\scriptsize $a_4$};
        \node at (9 + 0.6, 1 - 0.25) {\scriptsize $a_5$};

        \node (2) at (12, 0) {$a_2$};
        \node (3) at (13, 1) {$a_3$};
        \node (5) at (14, 0) {$a_5$};
        \node (1) at (12, 2) {$a_1$};
        \node (4) at (14, 2) {$a_4$};
        \draw[orange] (2) -- (3) -- (1);
        \draw[orange] (5) -- (3) -- (4);
    \end{tikzpicture}
    \caption{The Dyck path \textcolor{blue}{$d$} with \textcolor{red}{peaks of $b$} marked, and the associated poset \textcolor{orange}{$P_{d, b}$} on the right.}
    \label{fig:posetex}
\end{figure}

A \emph{linear extension} of a poset $P$ on $\{a_1, \ldots, a_n\}$ is a permutation $\sigma \in \s_n$ such that $\sigma_i < \sigma_j$ whenever $a_i <_P a_j$. 
We use $\mathcal{L}(P)$ to denote the set of linear extensions of $P$. 
For the poset $P$ from \Cref{fig:posetex}, we have $\mathcal{L}(P) = \{4 1 3 5 2, 4 2 3 5 1, 5 1 3 4 2, 5 2 3 4 1\}$.

For $d \in \Dyck_n$ and $b \in B_d$, the definition of $P_{d, b}$ gives us that
\begin{equation*}
    \mathcal{L}(P_{d, b}) = \{\sigma \in M_d \mid \Des(d, \sigma) = \Des(b, \delta_n)\}.
\end{equation*}
Also, since a Dyck path in $B_d$ is determined by the position of its peaks, these sets are mutually disjoint (and non-empty by \Cref{Dessetchar}). 
This gives us the following.

\begin{corollary}\label{LEplusLB}
    For any Dyck path $d$, we have the partition
    \begin{equation*}
        M_d = \bigsqcup_{b \in B_d} \mathcal{L}(P_{d, b}).
    \end{equation*}
    This gives the lower bound $m_d \geq |B_d|$.
\end{corollary}

From the proofs of \Cref{md,thm:vperm}, we have that for any Dyck path $d$, both $\bperm d$ and $\vperm d$ are in $\mathcal{L}(P_{d, \bounce d})$.

\begin{example}
    For $d = U^2DUD^2U^2D^2$, we have $B_d = \{b_1, b_2, d\}$ where
    \begin{equation*}
        b_1 = UD(U^2D^2)^2 \text{ and } b_2 = U^2D^2UDU^2D^2.
    \end{equation*}
    The corresponding posets are shown in \Cref{fig:LEpartex} and we have
    \begin{align*}
        M_d &= \mathcal{L}(P_{d, b_1}) \sqcup \mathcal{L}(P_{d, b_2}) \sqcup \mathcal{L}(P_{d, d})\\
        &= \{4 5 3 2 1\} \sqcup \{5 1 4 3 2, 5 2 4 3 1, 5 3 4 2 1\} \sqcup\{5 4 3 2 1\}.
    \end{align*}
\end{example}

\begin{figure}[H]
    \centering



    \begin{tikzpicture}[scale = 0.9]
        \node (5) at (12, -1) {$a_5$};
        \node (4) at (12, 0) {$a_4$};
        \node (3) at (12, 1) {$a_3$};
        \node (1) at (12, 2) {$a_1$};
        \node (2) at (12, 3) {$a_2$};

        \draw[orange] (5) -- (4) -- (3) -- (1) -- (2);

        \node at (12, -2) {$P_{d, b_1}$};
    \end{tikzpicture}\hspace{1cm}
    \begin{tikzpicture}[scale = 0.9]
        \node (5) at (12, -1) {$a_5$};
        \node (4) at (12, 0) {$a_4$};
        \node (3) at (11, 1) {$a_3$};
        \node (1) at (11, 2) {$a_1$};
        \node (2) at (10, 0) {$a_2$};

        \draw[orange] (5) -- (4) -- (3) -- (2);
        \draw[orange] (3) -- (1);

        \node at (11, -2) {$P_{d, b_2}$};
    \end{tikzpicture}\hspace{1cm}
    \begin{tikzpicture}[scale = 0.9]
        \node (5) at (12, -1) {$a_5$};
        \node (4) at (12, 0) {$a_4$};
        \node (3) at (12, 1) {$a_3$};
        \node (1) at (12, 2) {$a_2$};
        \node (2) at (12, 3) {$a_1$};

        \draw[orange] (5) -- (4) -- (3) -- (1) -- (2);

        \node at (12, -2) {$P_{d, d}$};
    \end{tikzpicture}
    \caption{Posets $P_{d, b}$ for all $b \in M_d$.}
    \label{fig:LEpartex}
\end{figure}
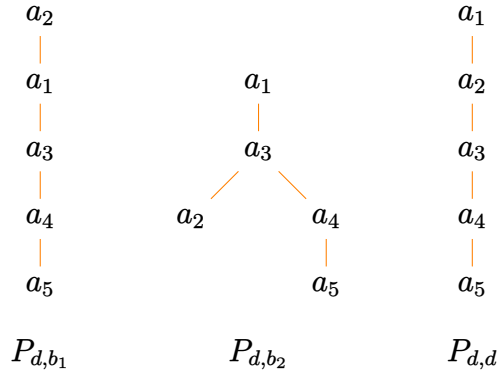

\begin{remark}
    Even though \Cref{LEplusLB} gives a method to generate the permutations in $M_d$, it is not easy to implement in practice. 
    This is mainly because we have to construct posets and compute their linear extensions. 
    Hence, it would still be preferable to find more direct algorithms such as \Cref{algo,valgo} to generate the permutations in $M_d$.
\end{remark}

An interesting problem is to characterize those Dyck paths $d$ such that $m_d = 1$. 
Although we do not solve this (see item \ref{md=1} in \Cref{sec:future}), we do characterize a related class of Dyck paths. 
Note that by \Cref{Dessetchar}, if $m_d = 1$, then we must have $|B_d| = 1$ (in fact $B_d = \{\bounce d\}$). 
We characterize when this happens.

\begin{proposition}\label{|Bd|=1}
    For any Dyck path $d$, we have $|B_d| = 1$ if and only if each peak of $\bounce d$ coincides with a peak of $d$. 
    The number of such Dyck paths graded by semilength is given by \cite[\href{https://oeis.org/A287709}{A287709}]{oeis}.
\end{proposition}

\begin{proof}
    To prove the necessity of the condition on peaks of $\bounce d$, we consider the reverse bounce path associated to $d$. 
    As mentioned before, this is the reverse of the bounce path associated to the reverse of $d$. 
    It can also be viewed as constructing the bounce path from the end of $d$ instead of the beginning. 
    If there is a peak of $\bounce d$ that doesn't coincide with a peak of $d$, then the reverse bounce path will be different from the bounce path. 
   The reverse bounce path of $d$ also lies in $B_d$. 
    Hence, if $|B_d| = 1$, we must have that the peaks of $\bounce d$ coincide with (some) peaks of $d$.

    In terms of bounce factors, the condition that each peak of $\bounce d$ coincides with a peak of $d$ corresponds to the fact that each bounce factor, apart from the last one, should end with an up-step. 
    Using similar ideas to the proof of \Cref{Dessetchar}, this implies that each peak of a Dyck path $b$ in $B_d$ must be at the end of a bounce factor (and hence $b = \bounce d$). 
    If there is a peak of $b$ not at the end of a bounce factor, it can be shown that $b$ will also have a peak in the last bounce factor, which would imply that it has more than $\bpk d$ peaks. 
    This contradicts the condition of lying in $B_d$.
    
    We now turn to counting these Dyck paths. 
    Note that for any $c = (c_1, c_2, \ldots, c_k)$ it can be shown that
    \begin{equation*}
        \prod_{i = 1}^{k - 1} \binom{c_i + c_{i + 1} - 2}{c_{i + 1} - 1}
    \end{equation*}
    is the number of Dyck paths $d$ such that $\bcomp d = c$ and each peak of $\bounce d$ matches a peak of $d$ (using the same ideas as in \Cref{cor: descreasing is maximizer}). 
    It can also be shown that this matches the number of Dyck paths where the maximum height in the path is attained by the first peak and there are $c_i$ up-steps at height $i$ for all $i \in [k]$ (see \Cref{bpk=ht}). 
    This proves that the Dyck paths we consider are counted by \cite[\href{https://oeis.org/A287709}{A287709}]{oeis} since the first comment in the OEIS entry states that the $n$-th term of the sequence counts Dyck paths of semilength $n$ where the maximum height in the path is attained by its first peak.
\end{proof}

Even though $m_d = 1$ implies that $|B_d| = 1$, the converse is not true. 
For example, $d = U(UD)^5D$ satisfies $|B_d|=1$ (since its bounce path has peaks that match peaks of $d$). 
However, $m_d = 2$ since $M_d = \{635241, 645231\}$.

There is another question which is related to determining the Dyck paths for which $m_d = 1$. 
For a given $d \in \Dyck_n$ and $b \in B_d$, when is $P_{d, b}$ a chain? 
That is, when is there a unique $\sigma \in M_d$ such that $\Des(d, \sigma) = \Des(b, \delta_n)$? 
Since we have already determined when $|B_d| = 1$, a solution to this question would determine when $m_d = 1$.



\section{Future directions}\label{sec:future}

In this section, we gather some questions that would be of interest to explore.

\begin{enumerate}
    \item The first obvious direction would be to extend the results in this paper to general canon permutations \cite{sergi2}. 
    Just as Dyck paths can be used to capture the ``shuffle order'' for a canon permutation with $m = 2$ copies of each entry. 
    For general $m$, this is done by rectangular standard Young tableaux with $m$ columns (see \cite[Figure 2]{sergi2}). 
    Given a standard Young tableau $T$ with $n$ rows and $m$ columns and a permutation $\sigma \in \s_n$, we construct the canon permutation $\can(T, \sigma)$ by placing $\sigma_i$ at the indices specified by the $i$-th row of $T$.

    We can now try to obtain similar results for polynomials of the form
    \begin{equation*}
        C_T(t) \coloneqq \sum_{\sigma \in \s_n} t^{\des(\can(T, \sigma))}.
    \end{equation*}
    The symmetry for such polynomials follows just as for the $m = 2$ case. 
    However, these polynomials are \emph{not} free of internal zeroes in general. 
    For example, if
    \begin{center}
        \begin{tikzpicture}[scale = 0.8]
            \draw (0, 0) grid (3, -3);
            \node at (-0.75, -1.5) {$T=$};
            \begin{scope}[shift = {(0.5, -0.5)}]
                \node at (0, 0) {$1$};
                \node at (1, 0) {$4$};
                \node at (2, 0) {$7$};
                \node at (0, -1) {$2$};
                \node at (1, -1) {$5$};
                \node at (2, -1) {$8$};
                \node at (0, -2) {$3$};
                \node at (1, -2) {$6$};
                \node at (2, -2) {$9$};
            \end{scope}
            \node at (3.25, -1.5) {,};
        \end{tikzpicture}
    \end{center}
    then $C_T(t) = t^2 + 2t^3 + 2t^5 + t^6$, which has an internal zero.
    
    What can be said about the degree of $C_T$? 
    A first step might be to find the appropriate generalization of (peaks of) bounce paths for rectangular standard Young tableaux.

    \item For a Dyck path $d \in \Dyck_n$, we have described two main methods (and generalizations of them) to choose a permutation from $M_d$. 
    We gather some observations about the permutations generated by these methods.

    For any $n$, we define
    \begin{align*}
        \CanDy(n) &\coloneqq \bigcup_{d \in \Dyck_n} M_d\\[0.1cm]
        \CanDy_b(n) &\coloneqq \{\bperm d \mid d \in \Dyck_n\}\\[0.2cm]
        \CanDy_v(n) &\coloneqq \{\vperm d \mid d \in \Dyck_n\}.
    \end{align*}
    One property that the second and third classes seem to satisfy is
    \begin{equation*}
        |\CanDy_b(n)| = |\CanDy_v(n)| = |\Dyck_{n - 1}|.
    \end{equation*}
    It also seems that $\bperm$ (respectively, $\vperm$) has different outputs for different primitive Dyck paths. 
    The above statements would imply that any Dyck path has the same image under $\bperm$ (respectively, $\vperm)$ as some primitive Dyck path. 
    For the class $\CanDy(n)$, the sequence of sizes does not seem to be in the OEIS \cite{oeis}. 
    The first few terms of the sequence are $1, 1, 3, 9, 34, 152, 771, 4371, \ldots$.
    
    Also, the number of Dyck paths $d$, graded by semilength, where $\bperm d = \vperm d$ seems to match \cite[\href{https://oeis.org/A5773}{A5773}]{oeis}. 
    Note that the generalization of \Cref{algo} mentioned in \Cref{genbperm} seems to give a relation between $\bperm d$ and $\vperm d$.
    
    \item \label{md=1}As mentioned in \Cref{sec: interpret leader}, one could try to characterize the Dyck paths $d$ for which $m_d = 1$. 
    These are Dyck paths that have a unique way to label them to obtain a canon permutation with the maximum possible descents. 
    The number of such Dyck paths graded by semilength seems to match \cite[\href{https://oeis.org/A88456}{A88456}]{oeis} (verified for semilength $\leq 9$). 
    A bijective proof could prove insightful (we expect that bounce paths would correspond to the non-decreasing sequences of the type counted by \cite[\href{https://oeis.org/A88456}{A88456}]{oeis}). 
    Characterizing when $m_d=1$ is also interesting because the sequence of coefficients could potentially arise, for example, as an Ehrhart $h^*$-polynomial, similarly to the descent polynomial of canon permutations \cite{danai}.

    \item As mentioned earlier, Elizalde proved \cite[Theorem 2.6]{sergi1} that for any $\sigma \in \s_n$,
    \begin{equation*}
    \sum_{d \in \Dyck_n} t^{\des(d, \sigma)} = t^{\des(\sigma)} N_n(t)
    \end{equation*}
    where $N_n(t) = \sum\frac{1}{n}\binom{n}{r}\binom{n}{r + 1}t^r$ is a Narayana polynomial. 
    Are there similar expressions for $C_d(t)$ in terms of well-known polynomials? 
    For the `simplest' Dyck paths, we obtain expressions in terms of Eulerian polynomials. 
    Setting $A_n(t) = \sum_{\sigma \in \s_n} t^{\des(\sigma)}$, it is straightforward to check that $C_{(UD)^n}(t) = A_n(t)$. 
    To compute $C_{U^nD^n}(t)$, we consider the multivariate polynomial
    \begin{equation*}
        \tilde{A}_n(t, u) \coloneqq \sum_{\substack{\sigma \in \s_n\\\sigma_n < \sigma_1}} t^{\des(\sigma)} + u\sum_{\substack{\sigma \in \s_n\\\sigma_n > \sigma_1}} t^{\des(\sigma)}.
    \end{equation*}
    This polynomial has been studied in \cite[Proposition 6.4]{cyclic} where it is shown that
    \begin{equation*}
        \tilde{A}_n(t, u) = t^{n - 1}f(t^{-1}) + uf(t) \text{ where }f(t) \coloneqq \frac{d}{dt}tA_{n - 1}(t).
    \end{equation*}
    Since $C_{U^nD^n}(t) = \tilde{A}_n(t^2, t)$, the result above shows that this polynomial too can be expressed in terms of Eulerian polynomials.

\end{enumerate}

\section*{Acknowledgements}
We thank Emil Verkama and Rob Ezno for their valuable commments and providing us with CanDy. 
DD was partially supported by the
Wallenberg Autonomous Systems and Software Program (WASP) funded by the Knut and Alice
Wallenberg Foundation, by the Spanish--German
project COMPOTE (AEI $PCI2024-155081-2$ and DFG $541393733$), and by the
Spanish project PID$2022-137283$NB$\-\ $C$21$ of MCIN/AEI/$10.13039/501100011033$. 
KM is supported by the Göran Gustafsson Foundation and the Verg Foundation. 
We are also grateful to FindStat \cite{FindStat}, OEIS \cite{oeis}, and SageMath \cite{sagemath}.

\bibliographystyle{abbrv} 
\bibliography{refs} 

\end{document}